\documentclass[11pt]{amsart}

\usepackage{amssymb}
\usepackage{amsthm}
\usepackage[curve,matrix,arrow]{xy}
\newenvironment{prfstab}{{\bf Proof of Theorem \ref{stabletomor}.}}{\hfill $\Box$ \medskip}
\newenvironment{prfabe}{{\bf Proof of Theorem  \ref{abeliandefectcon}.}}{\hfill $\Box$ \medskip}
\newenvironment{prfoutgen}{{\bf Proof of  Proposition  
\ref{Outgen}.}}{\hfill $\Box$ \medskip}

\theoremstyle{plain}\newtheorem{Theorem}{Theorem}[section]
\theoremstyle{plain}
\theoremstyle{plain}
\theoremstyle{plain}\newtheorem{Lemma}[Theorem]{Lemma}
\theoremstyle{plain}\newtheorem{Proposition}[Theorem]{Proposition}
\theoremstyle{definition}
\theoremstyle{definition}
\theoremstyle{definition}
\theoremstyle{definition}
\theoremstyle{definition}\newtheorem{Remarks}[Theorem]{Remarks}
\theoremstyle{definition}\newtheorem{Hypothesis}[Theorem]{Hypothesis}

\def\Gb{{\bf G}}

\def\Zb{{\bf Z}}

\def\Ub{{\bf U}}

\def\Aut{\mathrm{Aut}}

\def\Aut{\mathrm{Aut}}

\def\dim{\mathrm{dim}}

\def\Ker{\mathrm{Ker}} 
           
\def\Inn{\mathrm{Inn}}

\def\Out{\mathrm{Out}}
\def\mod{\mathrm{mod}}

\def\Soc{\mathrm{Soc}}

\newcommand{\GL}{\operatorname{GL}}

\begin{document}

\title {On blocks stably equivalent to  
a quantum complete intersection of dimension $9$ in characteristic 
$3$ and a case of the Abelian defect group conjecture} 
\author{Radha Kessar} 
\date{}
\address {Institute of Mathematics, University of Aberdeen, 
Fraser Noble Building, King's College, Aberdeen AB24 3UE, U.K. }

\begin{abstract}
Using a stable equivalence due to Rouquier, we prove that 
Brou\'e's abelian defect group conjecture holds for  $3$-blocks of defect $2$ 
whose  Brauer correspondent has a 
unique isomorphism class of simple modules. The proof  makes use of   the fact, 
also due to Rouquier,   that  a stable equivalence of Morita type between 
self-injective algebras  induces an isomorphism  between the connected  
components of the  outer automorphism  groups of the algebras.
\end{abstract}

\maketitle

\section{Introduction}
The abelian defect group conjecture  of 
Brou\'e \cite[Question 6.2]{Brou2},
\cite[Conjecture, page 132]{konzim}  states that  any $p$-block ($p$ a 
prime number)  of a 
finite group  $G$ whose defect groups are abelian is derived equivalent   
to its Brauer correspondent. The conjecture is known to hold
 when the defect groups in question  
are   cyclic  \cite{rickstab},    Klein $4$-groups
\cite{ricksplen}, and when the concerned  
blocks are nilpotent \cite{puignilp}. The conjecture has also  been proved  
in several situations under additional global  assumptions, that is,  
assumptions on  $G$. In this paper, we  prove the following result.

\begin{Theorem}\label{abeliandefectcon}  
Let $k$ be an algebraically   closed field of characteristic  $3$. 
Let  $G$ be a finite group,   $A$ a   block of $kG$,
$P$ a defect group of   $A$ and $C$ the  Brauer correspondent  of $A$  
in $kN_G(P)$.  Suppose  that  
$P  \cong C_3 \times C_3$ and that   $C$  is a non-nilpotent  block  with   
a unique isomorphism class of 
simple modules.   Then  $A$ and $C$ are Morita equivalent as $k$-algebras.  
\end {Theorem}

The  above theorem     is the first of its kind in that  it proves a case 
of  the abelian defect group conjecture  for all blocks   with 
a prescribed local structure which is of  wild representation type and 
which  is not obviously related  to  the nilpotent case.
We point out that we do not know if the equivalence 
constructed for the  proof of  Theorem \ref {abeliandefectcon}  is splendid or 
if it can be lifted to a complete local ring in characteristic zero.

The proof  of Theorem \ref{abeliandefectcon}   fits in with the 
inductive approach   to the abelian defect group conjecture outlined by 
Rouquier in \cite{rouquier-glue} :  the starting point is 
the  gluing result of  \cite{rouquier-glue}, which yields a
stable  equivalence of Morita type between the blocks  $A$ and $C$. 
Moreover,  given the assumptions on $C$ and  K\"ulshammer's 
structure results on blocks with normal defect groups  in  \cite{Ku}, 
it is known that  $C$ is Morita equivalent to  the quantum complete 
intersection  algebra $k \langle X, Y\rangle / \langle X^3, Y^3, XY + YX 
\rangle $.   Finally,  by  \cite{PuUs1}, it is known that $A$  and $C$ 
have isomorphic centers   and that $A$ also has a unique isomorphism 
class of simple modules. We obtain  
Theorem  \ref{abeliandefectcon} as a consequence  of the following   result.

\begin{Theorem} \label{stabletomor}  Let $k$ be an algebraically closed field 
of characteristic $3$.
Let $$B= k \langle X, Y\rangle /\langle X^3, Y^3, XY + YX \rangle $$  and  let 
$A$  be a
local, symmetric $k$-algebra. 
Suppose that  $\dim_k(A)= 9$, $Z(A) \cong Z(B)$ and that 
there is a stable equivalence of Morita type between $A$ and $B$. 
Then $A$ is isomorphic to $B$.
\end {Theorem}

The  proof of Theorem  \ref{stabletomor}   has two main steps. First, 
we use arguments in the style 
of  \cite{kulshammer:locsym}  to describe   $9$-dimensional 
local symmetric algebras whose  center  is isomorphic to  that of
$B$ - there are infinitely many isomorphism classes of these. 
In the second step, we show that  amongst the  algebras obtained in 
step 1, the only ones which are stably equivalent  (\'a  la Morita) to $B$ are 
isomorphic to $B$.  The crucial ingredient   here 
is the fact that  a stable equivalence of Morita type  
between self-injective algebras  preserves the connected component of   the 
outer automorphism group of the algebras (\cite{rouq:stab}).

\begin{Remarks} (i) By work of Kiyota in \cite{Kiy}, 
with the notation of Theorem \ref{abeliandefectcon}, it is known that $A$ 
has a unique  isomorphism class of simple module if and only if  $C$ 
has a unique  isomorphism class of simple modules. Thus Theorem  
\ref{abeliandefectcon}  remains true if one replaces the hypothesis that 
$C$ has a unique  isomorphism class of simple modules  by the hyothesis that 
either $A$ or $C$  has a  unique  isomorphism class of simple modules.

(ii) The proof of Theorem \ref{stabletomor}  is very computational; however  
it seems possible that  the methods   will find application in other 
situations, 
for instance, for      $p$-blocks  of defect $2$ 
whose Brauer correspondent  has one isomorphism    class of simple modules 
for   other  primes  $p$ (see \cite{HoKe}, \cite{Keslin09}).

\end{Remarks}

\section{Background Results on algebras and their automorphism groups}

Let $k$ be an algebraically closed field. All $k$-algebras  will be assumed 
to be finite dimensional  vector spaces over  $k$. Recall that   a  $k$-algebra 
$C$ is  symmetric if   there exists a $k$-linear function $s : C \to k$ 
such that  $s(xy)=s(yx) $ for all $x, y\in C$ and   no  non-zero left ideal of 
$C$ is contained in the kernel of $s$.    Also, recall that  
since  $k$ is algebraically  closed,  $C$ is local  if and only if 
$\dim_k(C/J(C))=1 $.

The following  gathers some well  known properties of local symmetric algebras;
we refer to \cite{kulshammer:habil}  for proofs.

\begin{Lemma}\label{locsymgen}  Suppose that $C$ is   a local symmetric 
$k$-algebra.  Then  

(i)  $ \dim_k \Soc(C)  =1 $. 

(ii) $\Soc(C)  \subseteq  \Soc(Z(C))$.   

(iii)$ \Soc (C) \cap [C, C] = 0$.  

(iv) $\dim_k(C)=  \dim_k(Z(C)) +\dim_k[C,C] $.

(v) $Z(C)$ is local and  $Z(C) \cap J (C) =J(Z(C))$. 

(vi)   If $n$ is the least natural number  such that $J^{n+1}(C) =0 $, then 
$\Soc(C)  = J^n(C)$.
\end{Lemma}

We will also use  the following results from \cite{kulshammer:locsym}.

\begin{Lemma}\label{kulsconv}\cite[Lemma E]{kulshammer:locsym}  
Let $C$ be a  $k$-algebra and let 
$I$ be a  two-sided ideal of   $C$. Let $m, n  $ be natural numbers  with 
$m\leq n $.
Suppose that
$$I^n = k-\text{span of \,} 
\{ x_{i1}\cdots x_{in} \, : \, i=1, \cdots, d \}  + I^{n+1} $$ 
with elements  $x_{ij} \in I$.  Then,
$$I^{n+m} = k-\text{span of \, } 
\{ x_{j1}\cdots x_{jm}x_{i1}\cdots x_{in} \, : \, i,j=1, \cdots, d \}  
+ I^{n+m+1}. $$ 
\end{Lemma}

\begin{Lemma}\label{kulsconv2}\cite[Lemma G]{kulshammer:locsym}   Let 
$C $ be a local symmetric $k$-algebra. If $n$ is  a natural number such that 
$\dim_k J^n(C)/J^{n+1}(C) = 1 $, then $J^{n-1}(C) \subseteq Z(C)$.
\end{Lemma}

For a finite dimensional $k$-algebra $C$, let  $\Aut(C)$  
be the group of automorphisms of $C$, viewed as  a (linear)  algebraic group.  
Clearly,   $\Aut(C)$ is a subgroup of  the algebraic group 
$\GL(C)$ of  the automorphisms of the $k$-vector space $C$.
For any    $\Aut(C)$-invariant  subspaces 
$ U, V$ of  $C$ with $ V \subseteq  U$, let   $ f_{U,V} \,  : \,  \Aut (C) 
\to    GL(U/V )$    denote  the   map  defined by    
$f_{U,V} (\varphi)  (u +V) =\varphi(u) + V$, $u \in U$.
Any power $J^i(C)$ of the radical      of $C$ is $\Aut(C)$-stable, hence
$\Aut(C)$ is  contained in  the parabolic subgroup  of  $\GL(C)$ of elements 
which stabilize  the flag $ C \supset  J(C) \supset J^2 (C)  \supset \cdots $.
For each $i \geq 0$, let  $f_i $ denote the map $f_{J^i(C), J^{i+1}(C)}$.

Let $\Inn(C) $ denote the  closed  connected normal subgroup of $\Aut(C)$  
consisting of inner automorphisms of $C$ and let   $\Out(C) $ 
be the quotient group    $\Aut(C)/\Inn(C)$.  Let   $\Aut^0(C)$ denote the  
connected component of the identity of $\Aut(C)$ and let 
$ \Out^0(C)= \Aut^0(C)/\Inn(C) $ denote the connected component of 
$\Out^0(C)$.

The following easy  results will come in handy for  calculating  
automorphism groups.
Recall that  for a connected  linear algebraic group  $\Gb $  over $k$,  
the  unipotent radical  $R_u(\Gb)$   of $\Gb $ is the largest  closed 
connected normal  unipotent subgroup of $\Gb$.

\begin{Proposition} \label{Outgen}  Suppose that $C$ is a local algebra.  Then, 

(i)  $\Ker(f_1) $ is a unipotent subgroup of $\Aut(C)$.

(ii)  $\Inn (A)  \leq \Ker(f_1) $.

(iii) There exists a  bijective  morphism of  algebraic  groups
$$ \hat f \, : \, \Out^0(C)/R_u(\Out^0(C)) \to  
f_1( \Aut(C))^0/R_u(f_1( \Aut(C))^0). $$ 
\end{Proposition}

The proof of the above is a  consequence of the following standard 
structural result.

\begin{Lemma} \label{alguse}  (i)   Let $\Gb $  be  a connected  
linear algebraic group  $\Gb $  over $k$.  Then $R_u(\Gb)$  contains  every 
normal unipotent subgroup of $\Gb$.

(ii) Let $\varphi : \Gb \to \Gb' $ be an 
epimorphism of connected algebraic groups  such that $\Ker(\varphi)$ is a 
unipotent subgroup of $\Gb$.   Then there is a bijective morphism
$$ \hat \varphi :  \Gb/R_u(\Gb)  \to  \Gb'/R_u(\Gb') $$
such that $\hat\varphi(gR_u(\Gb))=\varphi(g)R_u(\Gb') $ for  all 
$g \in \Gb$.

\end{Lemma}

\begin{proof} (i)  This  
follows  by \cite[Theorem 30.4 (b)]
{humph:book} and \cite[Exercise 26.14] {humph:book}.

(ii) Set $K= \Ker (\varphi)$.   By (i), $ K \leq R_u(\Gb)$  and by 
\cite[Corollary 21.3F,Theorem 30.4 (b), Exercise 26.14]{humph:book}, 
$\varphi(R_u(\bar G) ) =  R_u(\Gb' )$.  The result follows.

\end{proof}

\begin{prfoutgen} (i)   Since $ C $ is local, $\Aut(C)  =\Ker (f_0 )$. 
On the other hand, $\Ker(f_1)  \leq  \Ker( f_i)$,  for all  
$ i \geq 1 $.  Hence, if 
$\varphi \in \Ker(f_1)$, then  $ \varphi $  induces  the identity  on 
$J^i(C)/J^{i+1}(C)$ for all  $i \geq 0 $.   This proves (i).

(ii)    Let $ x $ be an invertible element of  $C$. Since $C$ is local, 
$ x = \lambda (1 + y )$ for some $0\ne \lambda \in k $ and  some $y \in J$.  
So $x^{-1} = \lambda^{-1} (1+ y') $, $y ' \in J$ and conjugation by $x$ 
induces the identity on $J(C)/J^2(C)$.

(iii)  By \cite[Proposition  7.4B]{humph:book},  
$f_1(\Aut^0(C))= f_1( \Aut(C))^0$.   So, by   (ii), the restriction of $f_1$ 
to   $\Aut^0(C) $  yields a surjective morphism
$$ \bar f_1 \, : \,   \Out^0(C) \to   f_1( \Aut(C))^0. $$
By (i), the  kernel of  $ f_1 $ is a  unipotent subgroup of $\Aut(C)$, hence 
the kernel  of  $\bar f_1 $ is a unipotent subgroup of $\Out^0(C)$.
So (iii) follows from Lemma   \ref{alguse} applied  with $\varphi $ equal  to 
$\bar f_1 $.

\end{prfoutgen}

For  $u, v$ elements of a  ring, $[u,v]$ denotes the commutator  $uv-vu $ of 
$u$ and $v$. 

\begin{Lemma} \label{geninn} Let $C$ be a  $k$-algebra such that 
$J^3(C) \subseteq  Z(C) $. For $ u \in J(C)$, let $\varphi_u : C \to C $ 
be the   automorphism defined by  $\varphi_u(v) =  (1-u) v (1-u)^{-1} $, 
$ v\in C$. 
Then, for  any $u,  v \in  J(C)$, $\varphi_u(v) = v + [v,u] + [vu, u] $.

\end{Lemma}

\begin{proof}  Let $ u, v \in  J(C)$ and suppose that $u^n =0 $. Then,
$(1-u)^{-1}= 1+ u  + \cdots + u^{n-1} $, and for any $v \in   J(C)$, 
$$ \varphi_u(v) = v + [v, u] + [vu, u] + \cdots+ [vu^{n-1}, u]. $$   
Since  $ J^3(C)  \subseteq Z(C)$, the result follows.
\end{proof}

\section {Preliminary Calculations}
For the rest of the paper, unless stated otherwise,  $k$ will denote an 
algebraically closed field of characteristic $p$ and 
$B =k\langle X, Y  \rangle /\langle X^3, Y^3, XY+YX \rangle $  be as in Theorem \ref{stabletomor}.  
Throughout this section $A$ will denote  a   $k$-algebra satisfying the 
following:

\begin{Hypothesis} \label{Ahyp1} $A$ is symmetric,  local, 
$\dim_k (A)  = 9 $, and $Z(A) \cong Z(B)$ as $k$-algebras.
\end{Hypothesis}

We fix  a symmetrizing form  $s \, : \, A \to k $ on $A$. For a subset 
$U$ of $A$ denote by $U^{\perp}$  the subset of  $A$ consisting of elements  
orthogonal to $U$ under the bilinear form associated to   $s$, that is 
$U^{\perp}$    is the set of elements $a' \in A$ such that  $s(aa')=0 $ 
for all   $ a\in U$.   For each $i \geq 0 $, let 
$J_i= J^i(A) $, $Z= Z(A)$,  $Z_i:=Z(A)\cap J^i(A)$.

\begin{Lemma}\label{centerA} Suppose  that $A$ satisfies Hypothesis
\ref {Ahyp1}. Then, $Z$ has a  $k$-basis $\{1,  z_i , 1\leq i\leq 5 \}$ 
such that  
$$ z_1z_2 = z_2z_1= z_5, \   z_iz_j =  0 
\text{ \, for all \, }  i, j  \text {\, s.t. \ }  \{i,j \} \ne \{1,2\}  $$
and  $\{z_i , 1\leq i\leq 5 \}$  is  a $k$-basis of  $Z_1$.
Further,  $\{z_i , 3\leq i\leq 5 \}$    is a $k$-basis  of 
$\Soc(Z(A)) $ and $\dim_k[A,A] =3 $.
\end{Lemma}

\begin{proof}   Since $Z(A) \cong Z(B)$,  in order to prove 
the first statement, it suffices to    prove an analogous statement for  
$Z(B)$.  Set $$ z_1':=X^2, z_2':= Y^2, z_3':= XY^2,  z_4':= YX^2, z_5':=  
X^2Y^2 \in Z(B). $$
Then  $\{1,  z_i' , i\leq i\leq 5 \}$  is a basis of $Z(B)$  and 
$ z_1'z_2' = z_2'z_1'= z_5' $  and all other pairwise  products   are $0$.    
Setting $z_i $ to be the image of $z_i $  under some isomorphism 
$Z(A) \cong Z(B)$ yields  the first statement.    
The  elements  $z_i $, $1\leq i \leq 5 $  are all nilpotent and central in 
$A$,  hence $ z_i \in  Z_1 $ for $ 1\leq i \leq 5 $. 
This proves the second assertion.  The   assertion  on $\Soc(Z(A)) $  
is  immediate from the multiplication of the $z_i'$ and the last assertion 
follows by Lemma  \ref{locsymgen}.
\end{proof}

\begin{Lemma}\label{elemfact2}
Suppose  that $A$ satisfies Hypothesis \ref {Ahyp1}.   For $ x\in J$ and 
$ i\in {\mathbb N}$,  denote by 
$C_{J_i}(x)$ the subring of $J_2$ of elements  of $J_2$ which commute with $x$.

(i)  $J_2 \nsubseteq  Z_2$.

(ii) $J_3 \subseteq Z_3 $.

(iii)  $\dim_k   (J_1/(Z_1+J_2))  =2 $ and  $\dim_k (J_2/ Z_2) =1 $.

(iv)   There exists a basis 
$$\{x+ Z_1+ J_2, y +Z_1+J_2 \} $$   of  $J_1/(Z_1+J_2) $ such that 
 $$C_{J_2}(x) =  C_{J_2}( y)  = Z_2. $$

(v)  $Z_1 $ is an ideal of  $A$.

(vi) For any $ u,v  \in J$, $uv +vu \in Z_2$.

(vii)  $ Z_1^2=\Soc(A)$.

(viii) $[A, A] \nsubseteq Z_1$ and  $ \Soc(Z(A)) = [A,A]\cap Z_1  
\oplus  \Soc(A) $.

(ix)  $J_5 =0 $, $ J_4=  Z_1^2=\Soc(A)$.

\end{Lemma}

\begin{proof} (i)   Note that since $A$ is local, any  subspace of $J_1$ 
which contains $J_2 $  is an ideal of $A$, and further  since 
$[A,A] \subseteq J_2$,  any such  subspace contains $[A,A]$.

Suppose  if possible that $J_2 \subseteq  Z_2 \subseteq Z_1$.  
Then  $Z_1$ and hence $Z_1^2 $ is an 
ideal of  $A$. In particular, 
$Z_1^2 \cap \Soc(A) \ne 0 $.  By   Lemma  \ref{locsymgen}, $\Soc(A)$ 
is $1$-dimensional and by  Lemma \ref{centerA}, $Z_1^2$ is $1$-dimensional. 
Hence, $Z_1^2 = \Soc(A)$, whence by  Lemma  \ref{locsymgen},
$[A, A] \cap  Z_1^2 =0 $.     On the other hand, since  
$ [A, A] \subseteq Z_1 $, and since 
$za=az $ for any $a \in A$, $z\in Z_1$,   
$$[A,A] Z_1  \subseteq [A,A] \cap  Z_1^2.  $$  Thus,  
$[A,A] Z_1=0 $, that is $[A,A]\subseteq \Soc(Z(A))$.
But  $$\dim_K ([A, A])= 3 = \dim_k (\Soc (Z(A))), $$  hence   
$[A, A] =  \Soc(Z(A)) $. 
This is a contradiction as by Lemma \ref{locsymgen} (ii), 
 $ \Soc(A) \subseteq   \Soc(Z(A))$  whereas  by Lemma 
\ref{locsymgen} (iii)  $[A, A] \cap \Soc(A) =0 $.
 
(ii)  By (i) and Lemma \ref{kulsconv2}, $\dim_k(J_3/J_4) \geq 2 $,
$\dim_k(J_2/J_3) \geq 2 $ and $\dim_k(J_1/J_2) \geq 2$. So, 
$\dim_k(J_4) \leq 2 $.   By Lemma \ref{locsymgen} (i),(vi), it follows that 
$\dim_k(J_4/J_5) \leq 1$. Hence  by Lemma \ref{kulsconv2},  $J_3 \subseteq Z$.

(iii) Let $c$ be the  codimension of $J_2 +  Z_1$ in $J_1$.  
Let  ${\mathcal A}$ be a 
basis  of  a complement of $J^2(A) +Z_1$ in $J(A)$ and  let  
${\mathcal A}'$ be a  (possibly empty)
basis  of  a complement of $J^2(A)$ in $J^2(A) +Z_1$ 
consisting of elements of $Z_1$. Then 
$\{1\} \cup {\mathcal A} \cup {\mathcal A}'  $ is a generating set  for 
$A$ as an algebra.   
If  the elements of $ {\mathcal A} $  pairwise commute, then 
$\{1\} \cup {\mathcal A} \cup {\mathcal A}' $   is a set of pairwise 
commuting generators of $A$, an impossibility as  $A$ is not  commutative. 
So  $c \geq 2 $.  On the other hand, the codimension of $Z_1$ in $J_1$ is $3$, 
so $c \leq 3$.  Suppose if possible that $c=3$.  Then 
$$\dim_k (J_2 +  Z_1) =5   =\dim_k (Z_1),$$  hence $J_2 \subseteq Z_1 $, 
a contradiction to (i).  Thus, $c=2 $. The second assertion is immediate from 
the first and (i).

(iv)    Let ${\mathcal A}=\{x,  y \} $ and ${\mathcal A}'$ 
be as in (ii). Since  $\{1\}   \cup {\mathcal A} \cup {\mathcal A}' $   
is a  generating set  for $A$ as an algebra,
$$ Z_2 = C_{J_2} (x) \cap  C_{J_2} ( y). $$ 
As $\dim _k Z_1=5 $ and  $x \in  C_{J_1} (x) $, 
$\dim_kC_{J_1} (x) \geq 6 $. Similarly, 
$\dim_k C_{J_1} ( y) \geq 6 $. On the other hand,
since $ \dim_kJ_1=8 $,  and $\dim _k Z_1=5 $,   at least one of  
$\dim_k C_{J_1} (x)$ and $\dim_k C_{J_2} ( y)$ is at most $6$.
Suppose that  $\dim_k C_{J_1} (x)=6$ and  $\dim_k C_{J_1} ( y)=7$.
Then by the same argument  applied to  $\{ x+  y,  y \}$, 
it follows that    $\dim_k C_{J_1} (x+ y)  =6 $.  Hence replacing
 $ y $ with $x+ y$, we may assume that 
 $$\dim_k C_{J_1} (x)=6  =  \dim_k C_{J_1} ( y).   $$
Thus, $C_{J_1} (x) $ is the $k$-span of 
$\{x, Z_1 \} $ and  $C_{J_2} ( y) $ is the $k$-span of 
$\{ y, Z_1 \} $.Consequently,  
$C_{J_2}(x)   \subseteq  Z_2 $ and 
$C_{J_2}( y) \subset  Z_2$.

From now on let  ${\mathcal A}=\{x,  y\}  \subseteq J_1$
 be such that $\{x+ Z_1+ J_2, y +Z_1+J_2 \} $   is a basis of  
$J_1/(Z_1+J_2) $  and $C_{J_2}(x) =  C_{J_2}( y)  = Z$.
 Let 
$ {\mathcal A}'  $  be a basis   of a complement of $J_2 $ in 
$J_2+Z_1$ consisting of elements of   $Z_1$.

(v)  Since $A$ is local,  and since   ${\mathcal A}'  \subseteq Z_1$, 
it suffices to prove that    $x Z_1  \subseteq Z_2 $ and   
$x Z_1  \subseteq Z_2 $. So, let
$ z \in Z_1 $.  Then,  $x z \in C_{J_2} (x)$, so   by the 
above $xz\in  Z_2 $.   Similarly,  $ y z \in Z_2 $.

(vi)  By (ii), $J_3\subseteq Z_3 $ and by (v), $Z_1 $ is an ideal 
of $A$.  Hence it suffices to prove that $x^2,  y^2 ,
x y + y x  \in Z_2$. Since $x^2  \in  C_{J_2}(x)$, $x^2 \in Z_2$. 
Similarly,  $ y^2 \in Z_2$.    Also, 
$$ x (x y + y x)  =x^2 y +x y x = 
 y x^2 +x y x=   (x y + y x) x $$
and  similarly,
$$  y (x y + y x)  =(x y + y x)  y, $$
so  $x y + y x  \in Z_2$.

(vii) By (v), $Z_1$  is an ideal  of $A$.  So,  $0\ne Z_1^2$ is an ideal of 
$A$, and hence   $\Soc(A)  \subseteq  Z_1^2$. But  since $ Z_1^2$  
is one dimensional, it follows that  $\Soc(A)  = Z_1^2$.

(viii)  Since $ Z_1([A,A]\cap Z_1)  \subseteq  [A,A]\cap  Z_1^2  $
 and  since by (vii), $  Z_1^2  =\Soc(A)$,  by Lemma \ref{locsymgen} (iii),
$$ ([A,A]\cap Z_1) Z_1= 0. $$  
Hence  $ [A,A]\cap Z_1  \subseteq \Soc(Z(A)) $ and again  by Lemma  
\ref{locsymgen} (iii), $ [A,A]\cap Z_1 \cap \Soc(A) =0 $.  The result follows  
since by (iii), the co-dimension of   $[A,A]\cap Z_1  $ in $[A,A]$ is 
at most $1$.

(ix)  Let $x_i \in {\mathcal A} \cup {\mathcal A}' $, $1\leq i\leq 5 $.  
If one  of  the $x_i$'s  is  in ${\mathcal A}' $, then   by  (vi) and (vii)
$$ x_1x_2x_3 x_4 x_5  \in Z_1J_3 J_1 \subseteq Z_1^2J_1=  0 .$$  
So, in order to show that $J_5=0 $, we may assume that  $x_i \in  {\mathcal A}$, 
$ 1\leq i \leq 5 $.  First suppose that  for some $i $, $1\leq i \leq 4 $, $x_i=x_{i+1} $, 
say $x_i=x_{i+1}=x $, By (v) and (vi),  for some $r,s \geq 0 $ such that $ r+s=5 $,
 $$ x_1x_2x_3 x_4 x_5 = x^r y^s \in Z_1^2 J_1= 0 .$$   
Suppose now that no two consecutive $x_i $'s are equal, so 
$$ x_1x_2x_3 x_4 x_5 =  xyxyx $$ or 
$$ x_1x_2x_3 x_4 x_5 =  yxyxy .$$
In the former case,  by (ii) and (v),
$$ x_1x_2x_3 x_4 x_5 =x^2yxy $$ and we are back in the previous situation.
The  second case is similar.

Suppose if possible that $J_4 =0 $. Then $J_3 =\Soc(A) $ is  $1$-dimensional, 
and by Lemma \ref{kulsconv2},  $J_2 \subseteq Z_1$, a contradiction to (i).
\end{proof}

\begin{Lemma} \label{nailrads}   Suppose that  A satisfies Hypothesis 
\ref{Ahyp1} and let     $\{ x+  Z_1+J_2,  y + Z_1+J_2 \} $ be a basis of 
$J_1/(Z_1+J_2) $.  Then, $\{xy + Z_2 \} $ is a basis of $ J_2/Z_2 $, 
$\dim_k[A,A] +J_3/J_3 =1 $, $\{ [x,y] +J_3\} $ is a basis of 
$ ([A,A] +J_3)/J_3 $ and  $\{ [x, y], [x,  xy], [y, xy] \} $  is a basis of $[A, A]$,

\end{Lemma}

\begin{proof}   
By  Lemma \ref{elemfact2} (v) and (vi), 
$ Z_1J_1 \subseteq Z_2 $  and $ x^2, y^2, xy+yx \in Z_2 $, and by Lemma \ref{elemfact2} (i), 
$J_2 \nsubseteq  Z_2$.      Hence, $xy \notin Z_2$ and it follows by   
Lemma \ref{elemfact2} (iii) that  $\{xy + Z_2 \} $ is a basis of $ J_2/Z_2 $.

By  Lemma \ref{elemfact2} (ii),     $J_3\subseteq Z$ and  by Lemma \ref{elemfact2} (viii) 
$[A,A] \nsubseteq   Z $. So $[A,A] \nsubseteq J_3 $.  On the other hand,
$[A, A] +J_3/J_3 $ is spanned by $[x,y]$.
So, $\{ [x,y] +J_3 \}$ is a basis of $[A, A]+J_3/J_3 $ and 
$\dim_k ([A, A]+J_3/J_3)= 1 $.    It follows from this that  $\{x, y, [x,y], Z_1\} $ spans 
$J_1$, and hence that 
$\{[x,y], [x,xy], [y,xy] \} $ spans $[A,A] $.
Since $\dim_k[A,A]=3 $,  $\{[x,y], [x,xy], [y,xy] \} $ is a basis of 
$[A,A]$.  

\end{proof}

\begin{Proposition}\label{firstreduction}  
Suppose that  A satisfies Hypothesis \ref{Ahyp1}.
Then one of the following holds.
\begin{equation}    \dim_k(J_1/J_2) = 3, \dim_k(J_2/J_3) =2,   
\dim_k(J_3/J_4) =2,  \dim_k(J_4)=1  \end{equation}
\begin{equation} \dim_k(J_1/J_2) = 2, \dim_k(J_2/J_3) =3,   
\dim_k(J_3/J_4) =2,  \dim _k(J_4)=1  \end{equation}
Moreover, $Z_1 J_1 \subseteq  J_3 $.
\end{Proposition}

\begin{proof}  As  in  the proof of   
Lemma \ref{elemfact2} (ii),   
$$\dim_k(J_3/J_4) \geq 2 , \   \dim_k(J_2/J_3) \geq 2,  \   
\dim_k(J_1/J_2) \geq 2. $$  By Lemma  \ref{elemfact2}, 
$J_4 = \Soc(A) $ is $1$-dimensional.
Thus, we have the following possibilities:
  $$\dim_k(J_1/J_2) = 3, \dim_k(J_2/J_3) =2,   
\dim_k(J_3/J_4) =2,  \dim_k(J_4)=1  $$
$$ \dim_k(J_1/J_2) = 2, \dim_k(J_2/J_3) =3,   
\dim_k(J_3/J_4) =2,  \dim _k(J_4)=1, $$
$$\dim_k(J_1/J_2) = 2, \dim_k(J_2/J_3) =2,   
\dim_k(J_3/J_4) =3,  \dim _k(J_4)=1 . $$
We show that the last is not possible.   Suppose, if possible that 
$J_1/J_2 $ has a basis $\{x +J_2, y+J_2 \}$ and that $\dim_k(J_2/J_3) = 2 $. 
Then $\{xy +J_3, yx+J_3, x^2+J_3, y^2+J_3 \}$ is a spanning set for $J_2/J_3$.
By Lemma \ref{elemfact2} (vi), $ x^2, y^2,  xy+yx  \in Z_2 $  and  by 
Lemma \ref{elemfact2}(i), 
$J_2 \nsubseteq Z_2$.   So, 
$\{xy +yx+J_3, x^2+J_3, y^2+J_3 \}$  spans  $Z_2/J_3 $.   Suppose  first that  
$x^2, y^2 \in   J_3 $.   Since $ Z_2 +J_3 $ has codimension $1$ in $J_2 $, 
$xy +yx \notin J_3 $.  So, replacing $ x$ with $x'=x+y $,  we may assume that 
$x^2 \notin J_3 $ and hence that $\{x^2 +J_3\} $  is a basis of $Z_2/J_3$. 
Thus, by Lemma  \ref{nailrads},   $\{xy+J_3 , x^2+J_3 \}$ is a basis
of $J_2/J_3 $. But then, by Lemma \ref{kulsconv}, $\{x^2y +J_4, x^3+J_4\} $ 
is a   spanning set  of $J_3/J_4 $, a contradiction. This proves the 
first assertion.

If   $\dim_k(J_1/J_2) =2 $, then  $Z_1 \subseteq J_2 $, 
and the second statement is trivial. Thus,  we may assume that  
$\dim_k (J_1/J_2) =3 $ and hence by the first assertion that  
$\dim_k(J_2/J_3) =2 $.
Let    $\{x +   J_2, y +J_2, z+J_2\}$ be a basis  of  $J_1/J_2 $ 
with $z \in  Z$  and suppose  if possible that  $Z_1J_1 \nsubseteq  J_3 $.   
Since  by Lemma \ref{elemfact2} (iii),  $Z_2$ has codimension $1$  in $J_2$, 
$\{x +  Z_1+ J_2 , y+Z_1+J_2 \} $ is a basis of $J_1/(Z_1+J_2)$.  By  
Lemma \ref{elemfact2}(viii), (ix), $z^2 \in \Soc(Z(A))=J_4 $,  
hence either 
$xz \notin J_3 $ or  $yz \notin  J_3 $, say $xz \notin J_3 $.   By Lemma \ref{elemfact2}(v), 
$xz \in  Z_2$ and  by   Lemma \ref{nailrads}, $\{xy + Z_2 \} $ is a basis of $ J_2/Z_2 $.
Thus,  $\{xy + J_3,  xz +J_3 \} $   is a basis of $J_2/J_3$. 
By Lemma \ref{kulsconv}, $\{x^2y + J_4,  x^2z +J_4 \} $
spans  $J_3/J_4$.  But by Lemma \ref{elemfact2}(vi), $x^2 \in Z_1$ whence by 
Lemma  \ref{elemfact2}(vii), (ix)
$$ x^2z\in Z_1^2 =\Soc(A) =J_4. $$
So $(\dim_k J_3/J_4 ) \leq 1 $, a contradiction. 
\end{proof}

\section{  The case  $\dim_k J_1/J_2 =3 $. }

In this section, we will work under  the following hypothesis.

\begin{Hypothesis} \label{Ahyp3} $A$ is symmetric,  local, 
$\dim_k (A)  = 9 $,  $Z(A) \cong Z(B)$ as $k$-algebras and 
 $\dim_k (J_1/J_2) =3 $.
\end{Hypothesis}

With the above hypothesis, by Lemma \ref{elemfact2}
$J_1/J_2 $ has a basis  of the form $\{ x+J_2, y+J_2, z + J_2  \} $,
 with $ z\in Z_1$.

\begin{Lemma} \label{3quant}
Suppose $A$ satisfies  Hypothesis \ref{Ahyp3}.
There exists a   basis  $\{ x+J_2, y+J_2, z + J_2  \} $  of 
$J_2/J_3$  such that  $z \in Z$,  $x^2\notin J_3 $, and 
$ xy +yx =0 $.
\end{Lemma}

 \begin{proof}   Let   $\{ x+J_2, y+J_2, z + J_2  \} $   be a basis of 
$J_2/J_3$   with  $z \in Z$.  By Proposition 
\ref{firstreduction}, $xz, zx \in J_3 $, so 
$\{ x^2+J_3, y^2+J_3, xy +J_3, yx+J_3 \}$  spans $J_2/J_3 $.   
By Lemma \ref{elemfact2} (vi), $ x^2, y^2, xy +yx \in Z_2$ and by Lemma  
\ref{elemfact2} (i), $J_2 \nsubseteq Z_2$.
Thus, $ \{x^2+J_3, y^2+J_3, xy +yx +J_3 \} $ spans  $Z_2/J_3 $.
If either $x^2 \notin J_3 $ or $y^2 \notin  J_3$, then interchanging
$x$ and $y$ if necessary, we have $x^2 \not\in J_3$. 
If both  $ x^2, y^2 \in J_3 $   then  
$xy +yx \notin J_3 $.  So, 
$$ (x+y)^2  = x^2 + y^2 + (xy +yx )  \not\in J_3 , $$
and replacing $x$ with $x+y  $  we may assume that $x^2 \notin J_3$.    Thus, since 
$ J_2  \nsubseteq Z_2$ and $\dim_k(J_2/J_3)=2 $, $\{x^2 +J_3 \} $ is a basis of 
$Z_2/J_3 $.

Since  $xy +yx \in Z_2 $, there exists  an $\alpha \in k $ 
such that
$$  xy + yx \equiv  \alpha x^2  \  \    \mod  J_3 $$ 
Set $ y'= y -\frac{1}{2} \alpha x  $. Then,
$$ xy' + y'x   \equiv 0 \   \  \mod   J_3 .$$  
Since  $\{ x+J_2,  y'+J_2, z + J_2  \} $ is  also  a basis of   $J_1/J_2 $, 
by replacing $y$ be $y'$, we may assume that $xy +yx \in J_3 $.

Since $\{x^2+J_3\} $ is a basis  of $Z_2/J_3$,  
by  Lemma \ref{nailrads},  $\{xy + J_3, x^2 +J_3\} $ is a basis  of $J_2/J_3 $, 
hence   by  Lemma \ref{kulsconv}, 
$\{ x^2y , x^3\} $ is a basis of $J_3/J_4$.
So, 
$$ xy +yx  \equiv \alpha x^3 + \beta  x^2 y  \ \mod  \  J_4 $$ 
for some $\alpha, \beta \in k $.

Set $ x'= x - \frac{1}{2} \beta x^2 $ and  $ y'= y- \frac{1}{2} \alpha x^2 $.
Then $\{ x'+J_2, y'+J_2, z + J_2  \} $  is   still a  basis of $ J_1/J_2 $, 
$ z \in Z$ , $x'^2 \notin  J_3 $ and    
 \begin{eqnarray*} x'y' + y'x'  &\equiv & (x - \frac{1}{2} \beta x^2)
( y- \frac{1}{2} \alpha x^2) +
( y- \frac{1}{2} \alpha x^2)  (x - \frac{1}{2} \beta x^2)  \\
 &\equiv &  xy +yx    - \alpha x^3  - \beta  x^2y  \\
&\equiv &  0 \  \mod \  J_4 
\end{eqnarray*}
Hence replacing $x$ with $x'$ and $y$ with $y'$, we may assume that 
$xy  +yx \in J_4 $.
Arguing as above, again   $\{ x^2y , x^3\} $ is a basis of $J_3/J_4$ and 
by applying  Lemma \ref{kulsconv} again,    $\{ x^3y , x^4\} $ spans $J_4$.
Since $\dim_k (J_4) =1 $, either $xy +yx = \alpha x^3y $ or
$xy +yx = \alpha x^4 $   for some  $\alpha \in k $.
Replacing $x$  by  $x - \frac{1}{2}\alpha x^3$  in the first case    
and  $ y$ by   $y -  \frac{1}{2}\alpha x^3 $  in the second case yields the 
result.
\end{proof}

\begin{Lemma} \label{3afterquant}  Suppose $A$ satisfies  
Hypothesis \ref{Ahyp3} and let  $\{ x+J_2, y+J_2, z + J_2  \} $   be a basis
$J_2/J_3$  such that  $z \in Z$, $x^2  \notin J_3$ and  
$ xy +yx =0 $. Then,

(i) For any $u \in Z_1$,  $uxy=xyu=0 $. In particular, 
$$ x^3y=yx^3=xy^3=yx^3= 0 . $$ 

(ii) $\{ xy +J_3, x^2+ J_3  \} $  is  a  basis of $ J_2/J_3 $,
$\{x^2y +J_4, x^3+ J_4 \} $  is  a  basis of $ J_3/J_4 $ and  
$\{ x^4 \} $ is  a  basis of $ J_4 $.

(iii)  $y^2 \not\in J_3 $.

\end{Lemma} 

\begin{proof}  (i) Let $u\in Z_1$. Since $ xy =-yx $, and since by 
Lemma \ref{elemfact2} (v), $ uy \in Z$, we have
$$ uxy= -uyx = -x(uy) = -x(yu)=  -uxy. $$  Thus, $uxy=0 $.  The second assertion follows from the first as by Lemma \ref{elemfact2} (vi),
$x^2, y^2 \in Z_2$.

(ii) The first assertion was   proved in the course  of the proof  of  
Lemma \ref{3quant}.  It follows from this and  Lemma \ref{kulsconv} that  $J_3/J_4$ is spanned by 
$\{x^2y +J_4, x^3 +J_4 \}$ and 
$J_4 $ is spanned by $\{x^3y, x^4 \}$. But by (i),  $x^3y =0 $. Hence 
$\{ x^4 \} $ is  a  basis of $ J_4 $.

(iii) By Lemma \ref{nailrads},  $\{[x,y], [x,xy], [y,xy] \}$ is a basis of 
$[A, A]$.   Since $xy=-yx$,  $[y,[xy]]=2xy^2$. If $y^2\in J_3 $, then   
$$[y,[xy]]=2xy^2 \in J_4 = \Soc(A),$$
a contradiction.
\end{proof}

\begin{Lemma} \label{3nailstr} Suppose $A$ satisfies  
Hypothesis \ref{Ahyp3}. Then there exists a basis 
$\{ x+J_2, y+J_2, z + J_2  \} $   of  $J_2/J_3$  such that    $z \in Z$, 
$\{x^2+J_3\} $ is  a basis of 
$Z_2/J_3$,   $ xy +yx =0 $  and such that  the following holds.

(i) $ zx=xz= 0 $.

(ii) $z^2= x^4 $.

(iii) $zy=yz=0 $.

(iv)  $y^2 =  x^2 +\alpha x^3 + \beta x^2y $ for some $\alpha, \beta  \in k$.
\end{Lemma}

\begin{proof}   Let   $\{ x+J_2, y+J_2, z + J_2  \} $   be a basis
$J_2/J_3$  such that  $z \in Z$, $\{x^2+J_3\} $ is  a basis of $Z_2/J_3$ and  
$ xy +yx =0 $.  Such  a   basis exists  by Lemma \ref{3quant}.

By Lemma  \ref{elemfact2}, $y^2 \in Z_2 $,  and by Lemma \ref{3afterquant}, $y^2 \notin J_3$,  so
$y^2 \cong \lambda x^2 \ \mod  \  J_3 $, for some $0\ne \lambda \in k$. 
Replacing  $y$ by  $\sqrt \lambda ^{-1}y$    does not affect  the relation 
$xy +yx =0 $, hence     we may assume  that $y^2 \cong x^2 \  \mod \ J_3 $. 
In particular,  $x^2y^2  = x^ 4 $. 

By   Proposition \ref{firstreduction}, $zx  \in J_3 $.
So, by   Lemma \ref{3afterquant},
 $ zx =\alpha x^3 + \beta x^2 y  + \gamma x^4 $  with 
$\alpha, \beta, \gamma \in k$.
By Lemma \ref{3afterquant}  (i), and by (i),
$$ 0 = zxy =  \alpha x^3y + \beta x^2 y^2  = \beta x^2 y^2  =  \beta x^4.$$
So, $\beta = 0 $.   Since  $ z':= z-\alpha x^2 -\gamma x^3 $   is an element 
of $Z$ and $z' \equiv z \ \mod \ J_2$,   
replacing $z$ by $z'$,  we may assume  that  $zx=xz=0$ and (i)  holds.

Since $\{x^2, z, x^2y, x^3, x^4 \}  \subseteq  Z_1$  and $\dim_k(Z_1) = 5$,
$\{x^2, z, x^2y, x^3, x^4 \} $ is a  basis of $Z_1$  by 
Lemma \ref{3afterquant}.     Now 
$ \{  x^2y, x^3, x^4 \} \subseteq    \Soc(Z(A)) $  and $\dim_k(\Soc(Z(A))=3 $, so 
$ z \notin        \Soc(Z(A))$.   Since  $zx^2=0 $, this means that  
$z^2 \ne 0 $. 
On the other hand,  by Lemma \ref{elemfact2} (vii),  $z^2 \in \Soc(Z(A))=\Soc(A)$. Thus, 
$ z^2 =\delta x^4 $, with $0\ne \delta \in k   $.   Replacing $z$ 
with a constant multiple  does not  affect  any of the  already established 
properties. Hence,
we may assume that  $ z^2 = x^4 $  and (ii) holds.

We now  show  (iii).   Again by Proposition \ref{firstreduction},   
$zy \in J_3$,  so  by Lemma \ref{3afterquant}, 
$$ zy  = \alpha x^3 + \beta x^2 y  + \delta x^4 , $$ with  
$\alpha, \beta, \delta \in k$.  By Lemma 
\ref{3afterquant},  
$$ 0 = zyx =  \alpha x^4, $$
whence $\alpha =0 $.  Since $zJ \subseteq  J_3$,  $ zJ_3=0 $,  and we have shown above that  $y^2 \cong x^2 \ \mod \ J_3 $.
Hence  by (i),
$$0= zx^2 = zy^2   =  \beta x^2 y ^2 = \beta x^4 . $$
But then $ \beta =0 $ and $ zy  = \delta  x^4 $.   Set
$ y'= y -  \delta z $.  Then  by (i) 
$$ y'x = yx =  -xy = -xy', $$  by (i), (ii)  
and Proposition \ref{firstreduction}, 
$${y'}^2    \equiv y^2 + \delta ^2 z^2  -2\delta yz 
\equiv  y ^2 \equiv  x^2 \mod \ J_3,$$ 
and by (ii),
$$ y'z = yz -\delta z^2 =  0. $$    Replacing $y$ with $y'$  
gives (ii).

Since $ y^2 \equiv  x^2  \ \mod  J_3 $, 
$ y^2 =  x^2 + \alpha  x^3 +\beta x^2y + \gamma x^4 $ for some 
$\alpha, \beta, \gamma \in  k$.

Set $ y'= y +\sqrt\gamma xy $.  Then $y'x +xy' =0 $, $y'z=zy' =0 $ and by 
Lemma \ref{3afterquant},  
$$ y'^2 = y^2   -\gamma x^2y^2 = x^2  +\alpha x^3 +\beta x^2y =  
x^2 +\alpha x^3 +\beta x^2y' . $$
So, replacing $y$ by $y'$ yields (iv).
\end{proof}

\begin{Proposition} \label{3nailstrcha}  Suppose $A$ satisfies  
Hypothesis \ref{Ahyp3}. Then there exists a basis 
$\{ x+J_2, y+J_2, z + J_2  \} $   of  $J_1/J_2$  such that    
$\{ xy+J_3, yx+J_3 \} $  is a basis of  of $J_2/J_3$, $\{ xyx+J_3, yxy+J_3 \} $ 
is a basis  of  $J_3/J_4 $,  $\{ xyxy \}$ is a basis of  $J_4 $ and such that 
the   following   relations hold.

(i)  $ zx=xz= zy=yz=0 $.

(ii) $z^2= xyxy $.

(iii) $y^2= x^2 =  \alpha xyx  + \beta yxy $,  $\alpha, \beta  \in k$.

(iv) $xyxy=yxyx$, 

(v)  For any  $u_i \in \{x, y, z \} $, $1\leq i \leq 5 $, $u_1u_2u_3u_4u_5=0 $.

Moreover, the above is a complete set of generators and relations for the 
algebra  $A$, that is  $A \cong  k\langle x, y, z \rangle/I $, where 
$I$ is the ideal of $k\langle x, y, z \rangle $,   generated by: 

$ \{ zx, xz, zy, yz,   z^2 -  xyxy,  y^2-x^2,  
y^2-\alpha xyx +\beta yxy \} \cup \{u_1u_2u_3u_4u_5, \ u_i \in \{x, y, z\}, 
1\leq i \leq 5 \}, $   $\alpha, \beta   \in k $.
\end{Proposition}

\begin{proof}  Let $\{x+J_2, y+J_2, z+J_2 \}$ be a basis    of $J_1/J_2 $ 
satisfying the conditions of Lemma \ref{3nailstr}.  Let $i$ be a primitive
$4$-th root of unity in     $k$ and set 
$ x' =x +iy  $, $ y'=x-iy $, $z' = z $.    Then $ x'^2=y'^2 \in J_3 $ and 
$x'z'=z'x'=z'y'=y'z'=0 $.   A  monomial of   two terms  in 
$x', y',  z'$  which involves  any of  
$x'^2 $, $y'^2 $ or $z'$ is in $J_3 $, hence $\{x'y'+ J_3, y'x' +J_3 \} $ spans 
$J_2/J_3 $ and  is therefore a basis of $J_2/J_3$.  Similarly, a  monomial of 
three  terms  in 
$x', y',  z'$ which  involves  any of  
$x'^2 $, $y'^2 $ or $z'$  is in $J_4 $, hence $\{ x'y'x' +J_4, y'x'y' +J_4\} $
is a basis of $J_3/J_4 $ and  $J_4 $ is    spanned by 
$x'y'x'y' =y'x'y'x'$.  Write
$$ x'^2= y'^2=  \alpha'x'y'x' + \beta' y'x' z'  + \delta x'y'x'y'. $$
Replacing  
$x' $ by $x'- \frac{\delta}{2} y'x'y'$ and $y' $ by 
$y'- \frac{\delta}{2} x'y'x' $  yields 
$$ x'^2= y'^2=  \alpha'x'y'x' + \beta' y'x' y'. $$ 
Since $z'^2 $ is a non-zero element of   $ J_4 $,  replacing 
$z' $ by  $\delta z' $ for  a  suitable $\delta $, we may assume that 
$z'^2=  x'y'x'y'$.    Thus, $\{x', y' ,z'\}$ satisfy the relations (i) -(v). 
The final  assertion  follows easily from  this and the fact that 
$\dim_k(A)=9 $.
\end{proof}

For the sake of completeness we record  the following without proof.

\begin{Proposition} \label{3afterstrrmk}    For any $\alpha, \beta \in k$,  
the algebra  $k\langle x, y, z \rangle/I $, where 
$I$ is the ideal generated by 
$ \{ zx, xz, zy, yz,   z^2 -  xyxy,  y^2-x^2,  
y^2-\alpha xyx +\beta yxy \} \cup \{u_1u_2u_3u_4u_5, \ u_i \in \{x, y\}, 
1\leq i \leq 5 \} $  
is a local symmetric $k$-algebra of dimension $9$ and with center 
isomorphic to $Z(B)$.
\end{Proposition}

\section{ The case $\dim_k(J_1/J_2) =2$.  }

The   following lemma    divides  the case that  $\dim_k(J_1/J_2)=2 $ 
in two subcases.

\begin{Lemma} \label{2start}   Suppose that  $A$ is symmetric,  local, 
$\dim_k (A)  = 9 $,  $Z(A) \cong Z(B)$ as $k$-algebras and 
$\dim_k (J_1/J_2)=2 $.
Then there exists a basis $\{x+J_2, y +J_2 \}$  of $J_2/J_3$ 
such that  either  $xy +yx \in J_3 $ or $ y^2 \in J_3 $. 
\end{Lemma}

\begin{proof} By Proposition \ref{firstreduction}, 
$\dim_k(J_2/J_3)=3 $, and  by   Lemma \ref{elemfact2}, 
$$\dim_k(Z_2/J_3)=2, \  
\dim_k(J_2/Z_2) =2 .$$ Let $\{x+J_2, y +J_2 \}$  be a  basis of $J_1/J_2$.  
Then,
$\{xy +J_3, yx+J_3, x^2+J_3, y^2 +J_3 \} $ is a spanning set for 
$J_2/J_3 $.    By Lemma \ref{elemfact2}, 
$\{xy +yx, x^2,  y^2 \} \subseteq  Z_2 $ and $J_2 \nsubseteq Z_2$,  so   
$\{xy +yx +J_3, x^2+J_3, y^2+J_3 \} $ spans $Z_2/J_3 $ and  
some two element subset    of  
$\{xy +yx +J_3, x^2+J_3, y^2+J_3 \} $ is a  basis  of $Z_2/J_3 $. 

Suppose 
first that   $\{x^2+J_3, y^2+J_3 \} $ is a  basis  of $Z_2/J_3 $. Let
$ xy + yx \equiv \lambda x^2  + \mu y^2 \  \mod  \ J_3 $,
$ \lambda, \mu \in k $.  If  $\lambda =\mu =0 $, then  $xy +yx \in J_3 $.
So, suppose that $\lambda \ne 0 $.  By  replacing $y$  with 
$\lambda^{-1} y $, we may  assume that $ \lambda =1 $.    First 
consider the case  that $\mu \ne  1 $. 
Set 
$$\sigma = -1 + \sqrt{1-\mu}, \  \tau = -1 - \sqrt{1-\mu}, \   
x'= x +\sigma y,   \   y '= x  + \tau y .$$
Then  since $ \sigma \ne \tau $,  
$\{ x' +J_2, y' +J_2 \} $  is  a basis of $J_1/J_2$ and
$$ x'y' + y'x' \equiv 0 \  \mod  \ J_3 . $$ 
So, replacing $\{x, y\} $ with $\{x', y' \}$ proves the result.
Now  consider the case  $\mu =1 $ and set $ y'=y-x $. Then  
$$ y'^2 = y^2 + x^2 -(yx +xy )  \in J_3. $$   Replacing  $ y $ by 
$y'$ yields the result.  

Now suppose that 
$\{xy +yx +J_3, x^2+J_3 \} $ is a  basis  of $Z_2/J_3 $ and let 
$$ y^2\equiv \alpha x^2 + \beta (xy +yx ) \  \mod \ J_3$$ for 
$\alpha, \beta \in k$.   Set $y'= y -\beta x$. Then,
$$ y'^2 =      y^2 +\beta^2 x^2 -\beta (xy +yx) \equiv  
(\alpha +\beta^2) x^2  \  \mod \ J_3. $$ 
So, replacing $ y$ by $y'$ we may assume that  
$$ y^2 \equiv \gamma  x^2 \ \mod \ J_3 ,  \      \gamma \in k $$
If $\gamma =0 $, the result is proved. Otherwise, replacing $y $ by 
$\gamma^{-1}  y $, we may assume that 
$$ y^2 \equiv  x^2 \ \mod \ J_3 .$$
Now set $x'= x+y $ and $ y'=x-y $. Then $ \{ x' +J_2, y' +J_2\} $ 
is a basis of $J_1/J_2 $ and 
$$ x'y' + y'x'   \in J_3. $$ 
So, replacing  $ \{ x +J_2, y +J_2\} $ with  $ \{ x' +J_2, y' +J_2\} $ 
yields the result.
\end{proof}

In view of the above Lemma, for the rest of the section, 
we will work under one of the   following two hypotheses.

\begin{Hypothesis} \label{Ahyp2nice}  $A$ is symmetric,  local, 
$\dim_k (A)  = 9 $,  $Z(A) \cong Z(B)$ as $k$-algebras 
$\dim_k (J_1/J_2)=2 $ and   $J_1/J_2$   has a basis $\{x+J_2, y +J_2 \}$ 
with  $xy +yx \in J_3 $.
\end{Hypothesis}

\begin{Hypothesis} \label{Ahyp2bad}   $A$ is symmetric,  local, 
$\dim_k (A)  = 9 $,  $Z(A) \cong Z(B)$ as $k$-algebras,
$\dim_k (J_1/J_2)=2 $ and   $J_1/J_2$   has a basis $\{x+J_2, y +J_2 \}$ 
such that  $ y^2 \in J_3 $.
\end{Hypothesis}

\begin{Lemma}\label{2nicequant}  Suppose that $A$ satisfies Hypothesis 
\ref{Ahyp2nice}. Then   $x$ and $y$ may be chosen  such that   $xy +yx = 0$.
\end{Lemma}

\begin{proof}   Since $xy+yx \in J_3 $,
$[x,[xy]] \equiv    2x^2y \  \mod  \  J_4 $ and 
$[y,[xy]]   \equiv    -2xy^2 \  \mod \ J_4 $.  On the other hand, 
by Lemma \ref{nailrads},  and the fact that $[A,A] \cap  J_4 =0 $,
$\{[x,[xy]] +J_4, [y,[xy]] +J_4\} $  is a basis of   $J_3/J_4 $. Thus, 
$\{ x^2y +J_4, xy^2+J_4 \} $ is a   basis  of $ J_3/J_4 $.   
So,   $$  xy + yx \equiv \lambda x^2 y + \mu xy^2 \  \mod  \  J_4. $$
Set $ x' = x- \frac{\lambda}{2} x^2 $ and   $ y' = y- \frac{\mu}{2} y^2 $. 
Then $ \{ x', y' \} $  is a basis of $J_1/J_2$ and  by Lemma \ref{elemfact2},
$$ x'y' + y'x' \equiv (x- \frac{\lambda}{2} x^2) (y- \frac{\mu}{2} y^2) + 
( y- \frac{\mu}{2} y^2) (x- \frac{\lambda}{2} x^2) \equiv 0 \ \mod  \ J_4  $$
So,  replacing $\{x, y \} $ with $ \{x', y' \} $, 
we may assume that $xy +yx \equiv 0 \  \mod \ J_4  $.  By the first 
part of the argument, 
$\{ x^2y +J_4, xy^2+J_4 \} $ is a   basis  of $ J_3/J_4 $.    
Hence,  by Lemma \ref{kulsconv} 
$xy+yx =\lambda x^2y^2 $ or $  xy +yx= \lambda x^3y $  for some 
$\lambda \in k $.   In the  first case,
replacing  $y$ by  $ y':= y -\frac{\lambda}{2} xy^2 $ and in the second  case 
replacing   $y $ by $ y':= y -\frac{\lambda}{2} x^2y $   yields the result.

\end{proof}

\begin{Proposition} \label{2nicequnail}  Suppose that $A$ satisfies Hypothesis 
\ref{Ahyp2nice}. Then there exists a basis $\{x+J_2, y+J_2 \} $ of 
$ J_1/J_2 $ such that  $\{x^2+J_3, y^2+J_3 , xy+J_3 \} $ is a basis of   $J_2/J_3 $, 
$\{x^2y +J_4, xy^2+J_4 \} $ is a basis of   $J_3/J_4 $, $\{x^2y^2 \} $ is a basis of 
$J_4 $, and $J^5=\{0\} $ and such that the following relations hold.

(i) $xy +yx =0 $.

(ii) $ x^3 = \alpha x y^2 +\beta x^2y^2  $  and 
$y^3 = \gamma  x^2y  +\delta x^2y^2 $   for some 
$\alpha, \beta, \gamma, \delta \in k $, with $\alpha, \gamma \in \{ 0, 1 \}$.

(iii)  $ x^3y= xy^3= 0$.

(iv)For any  $u_i \in \{x, y \} $, $1\leq i \leq 5 $, 
$u_1u_2u_3u_4u_5=0 $.

The above is a complete set of generators and relations for the 
algebra  $A$, that is  $A \cong  k\langle x, y,  \rangle/I $, where 
$I$ is the ideal of $k\langle x, y  \rangle $,   generated by: 

$ \{ xy +yx,   x^3 - \alpha x y^2 -\beta x^2y^2 , y^3 -\gamma  x^2y  
-\delta x^2y^2, x^3y, xy^3  \} \cup \{u_1u_2u_3u_4u_5, \ u_i \in \{x, y\}, 
1\leq i \leq 5 \}, $  $\alpha, \gamma \in \{0, 1\} $,  $\beta, \delta \in k $.

\end{Proposition}

\begin{proof}   By Lemma \ref{2nicequant}, there is a basis 
$\{x+J_2, y+J_2 \} $  of $ J_1/J_2 $ with $xy +yx =0 $.  
Since $x^3, y^3 \in J_3 \subseteq Z $  by Lemma \ref{elemfact2}, 
the relation $xy +yx =0 $ forces
$x^3y = -yx^3 $ and $xy^3=-y^3 x $. Thus, $x^3y =xy^3 =0 $.  
As noted   in the  proof of  Lemma \ref{2nicequant},  
$\{x^2y+J_4, xy^2 +J_4\} $ is  a basis of $J_3/J_4 $  and   either 
$\{ x^3y \}$ or 
$\{x^2y^2 \}$ spans $J_4 $.  But,  $x^3y=0 $. Thus, $x^2y^2 $ is a basis of 
$ J_4 $. Write
$$ x^3 = \lambda  x^2y  +   \alpha  xy^2 + \beta  x^2y^2 . $$
Multiplying on the  right with $y$ yields
$$ 0 =  \lambda x^2y^2 $$  hence $\lambda  =0 $. Similarly,
$$ y^3 =   \gamma x^2y  + \delta x^2y^2 .$$
If $\alpha =0 $, set $y'= y $, and if $\alpha \ne 0 $, set 
$y'= \sqrt{\alpha^{-1}} y$.  If  $\gamma =0 $, set $x'= x $, and if 
$\gamma \ne 0 $, set 
$x'= \sqrt{\gamma^{-1}}x $. Then replacing $\{x, y \}  $ with $\{x', y' \}$ 
gives $\alpha, \gamma \in \{ 0, 1 \}$.
Thus (i)-(iv) are satisfied.   The final assertion follows from the fact 
that   any algebra satisfying (i)-(iii)  has dimension at most $9$.
\end{proof}

For $A$  satisfying Hypothesis \ref{Ahyp2bad},  we  will 
require only  partial  structure results.

\begin{Lemma}  \label{2bad}
Suppose that   $A$ satisfies hypothesis  \ref{Ahyp2bad}.  Then,

(i) $\{x^2, xy, yx \} $ is a basis of $J_2/J_3 $. 

(ii)  $yxy \notin J_4 $.

(iii) $\{xyxy  \}$    is a basis  of $J_4$.

(iv) $\{ xyx +J_4,  yxy +J_4 \}$ is a basis of  $J_3/J_4 $.
\end{Lemma}

\begin{proof}  (i) is immediate from the first part of the proof of  
Lemma \ref{2start}. By Lemma \ref{nailrads},
$\{ [xy, x], [xy, y] \} $  is a basis of $J_3/J_ 4 $. But 
$ [xy, y] =yxy-y^2x  \equiv yxy   \  \mod  \ J_4 $. This proves (ii).
Suppose that $xyxy = 0 $.  Since $ yxy \in Z$, this also means that 
$ yxyx=0 $. 
Since  $y^2\in J_3 $,  
$y yxy =0= yxyy $.   Thus $ yxy \in   \Soc(A)=J_4 $, contradicting  (ii).  Thus, 
$xyxy \ne 0 $, proving (iv). 
Since $\dim_k(J_3/J_4 )=2 $,   by (ii), in order to prove   
(iv) it suffices to show that     $xyx  +J _4 \ne \lambda yxy +J_4 $ for 
$\lambda \in k $. If this were the case then 
multiplying  on both sides on the right with  $y$ would give
$xyxy= \lambda yxyy= 0 $ and this  would contradict (ii) as  $yxyx=xyxy$.
\end{proof}

\begin{Lemma} \label{2getboring}    Suppose that $A$ satisfies Hypothesis 
\ref{Ahyp2bad}.    Then,  $x$ and $y$ can be chosen 
such that  

(i)  $x^2y \in J_4 $.

(ii) $x^3y=0 $.

(iii) $x^3   \equiv  yxy  \  \mod \  J_4 $.
\end{Lemma}

\begin{proof}   Let $ \{x +J_2, y +J_2 \} $   be  a basis of  
of $J_1/J_2$ with $y^2\in J_3 $.  By Lemma \ref{2bad}, we may write
$$ x^2y \equiv \alpha xyx +\beta yxy  \ \mod  \ J_4. $$ 
Then, $$ 0= x^2y^2 = \alpha xyxy . $$ 
Hence $\alpha =0 $. Now 
$$ (x- \beta y )^2 y \equiv  x^2y - \beta yxy  \equiv 0 \  \mod  \ J_4 $$  
so replacing   $x$ by $x'= x- \beta y $  we may assume that 
(i) holds. (ii)  is immediate from (i).  

By  Lemma \ref{2bad},
$$x^3 \equiv   \alpha xyx +  \beta yxy   \ \mod  \in  J_4 . $$  
Multiplying with $y$ on the right yields  $\alpha =0 $.  
Now suppose if possible that $\beta =0 $.  
Then multiplying with $x$ yields  $x^4 =0 $. By  (ii) 
$$x^2(xy +yx) = 2x^3y =0 .$$   Since 
$\{ xy+yx, x^2, xyx, yxy, xyxy \}$ is a basis of $Z_1=J(Z)$  and since  
$\{xyx, yxy, xyxy \}  \subseteq  \Soc(Z(A)$, it follows that 
$ x^2 \in \Soc(Z(A))$, and hence that $ \dim_k(\Soc(Z(A))  \geq 4 $, 
a contradiction. So 
$\beta \ne 0 $. Replacing $x$   by $x':=  \sqrt {\beta^{-1}}x$
yields a basis      satisfying (i), (ii), (iii). 

\end{proof}

\section{Outer automorphism groups}

In this section we will keep to the notation   introduced 
for outer automorphism groups  in  Section 2.

\begin{Proposition}\label{out3}      Suppose that $A$ satisfies  
Hypothesis \ref{Ahyp3}, and let $\alpha, \beta  $ be as in  
Proposition \ref{3nailstrcha}.

(i)  If either $\alpha $ or $\beta  $ is non-zero, then 
$\Out^0(A)/R_u(\Out^0(A)) $ is  contained in  a 
one-dimensional   torus.

(ii)   If $\alpha=\beta =0 $, then 
$  \Out^0(A)/R_u(\Out^0(A))$  is a two-dimensional torus.
\end{Proposition}

\begin{proof}  Let $\{x+J_2, y+J_2, z+J_2\} $ be  a basis of  
$J_1/J_2 $  as in  Proposition \ref{3nailstrcha}.  
  Let  ${\mathcal B}$ be the ordered basis 
$\{x, y, z, xy, yx, xyx, yxy, xyxy \}$ of $J_1$ and let  
$\varphi :  A \to A$ be a $k$-linear map. For each  $u \in  {\mathcal B}$, 
write $\varphi (u) =\sum_{v \in {\mathcal B} } \lambda_{v, u} v $.     
Suppose that $\varphi \in \Aut(A)$.
Then, $$ 0 \equiv \varphi(x)^2   \equiv \lambda_{x,x}  \lambda_{y, x} (xy +yx)  
\  \mod  \  J_3, $$
$$0  \equiv \varphi(y)^2 \equiv   \lambda_{x,y}  \lambda_{y, y} (xy +yx)  
\    \mod  \  J_3, $$
$$ \lambda_{x,x}  \lambda_{y, x} = \lambda_{x,y}  \lambda_{y, y} =0, $$
that is  either $\lambda_{x,y} =\lambda_{y,x}=0 $ or 
$\lambda_{x,x} =\lambda_{y,y} =0 $.
Identifying $GL ( J_1/(J_2  +Z_1)  $ with     $GL_2(k)$    through the  
ordered basis $\{x  + J_2+Z_1 ,  y + J_2 +Z_1 \} $    it follows  that  the 
connected component   of $f_{J_1, J_2 +Z_1} (\Aut(A) )^0 $  is   contained    
in the subgroup of diagonal matrices. Hence, if $ \varphi    
\in \Aut^0(A)$ then  
\begin{equation}\label{3diag1}   
\lambda_{x,y} =\lambda_{y,x}=0. \end{equation} 
We will assume  from now on  that  this is the case.

Since $\varphi (z) \in Z(A)$, 
\begin{equation}\label{3diag2}    \lambda_{x,z}=\lambda_{y,z}=0  
\end{equation}
So, 
$$0 \equiv  \varphi(z) \varphi(x) \equiv   
\lambda_{xy,z} \lambda_{x,x}xyx     \ \mod  \  J_4, \text{\ and \ }  
0 \equiv  \varphi(z) \varphi(y) \equiv   
\lambda_{yx,z} \lambda_{y,y}yxy     \ \mod  \  J_4,    $$  
\begin{equation} \label{3zcon1} \lambda_{xy,z}=  \lambda_{yx, z} =0. 
\end{equation}  

So, $$  \lambda_{x,x}^2 \lambda_{y,y} ^2 xyxy= \varphi(x)\varphi(y)
\varphi(x)\varphi(y)=   \varphi(z)^2 =  \lambda_{z,z}^2  xyxy,  $$
yielding
$$ \lambda_{z,z}^2 =  \lambda_{x,x}^2 \lambda_{y,y}^2 .$$  Thus,  if  
$\varphi \in  
\Aut^0(A)$  then  
\begin{equation} \label{3zdiag} 
\lambda_{z,z} = \lambda_{x, x} \lambda_{y,y} . \end{equation}

Identify $GL(J_1/J_2) $ with $GL_3(k)$ via the ordered basis 
$\{ x+J_1, y+J_1, z+J_1 \} $  of $ J_1/J_2 $.  By Equations
\ref{3diag1} and  \ref{3zdiag}, $f_1 (\Aut^0(A)) $ is contained in 
the  subgroup  of lower triangular matrices  $(a_{ij}) $ 
satisfying $a_{33} = a_{11} a_{22} $, $ a_{2,1}=  0 $.   So 
$R_u(f_1 (\Aut^0(A))$ is the subgroup of  $f_1 (\Aut^0(A)) $ consisting 
of  matrices for which $a_{ii} =1 $, $ 1\leq i \leq 3 $.

 (i)   Suppose first that  $\beta \ne 0 $.
Comparing  the coefficient of   $yxy $  on both sides of  the relation   
$\varphi(x) ^2 = \alpha \varphi(x) \varphi(y) \varphi(x) + \beta  
\varphi(y) \varphi(x) \varphi(y)  $ yields
$$ \lambda_{x,x}^2  \beta = \lambda_{x,x} \lambda_{y,y}^2 \beta, $$
hence, $\lambda_{x,x} =  \lambda_{y,y}^2$.  It follows  from the 
above discussion that $f_1 (\Aut^0(A))/  R_u((\Aut^0(A))$   is   
isomorphic to  a subgroup  of   the group of   diagonal matrices  of  
$GL_3(k)$  satisfying $a_{11} = a_{2,2} $ and $ a_{3,3}  = a_{2,2}^ 3 $. 
So, the result follows by Proposition \ref{Outgen}. Note that $f_1 (\Aut^0(A))=
f_1 (\Aut(A))^0 $ (see \cite[Proposition  7.4B]{humph:book}).

Similarly,  if $\alpha \ne 0 $, then  
comparing  the coefficient of   $xyx $ in the 
relation $\varphi(y) ^2 = \alpha \varphi(x) \varphi(y) \varphi(x) + \beta  
\varphi(y) \varphi(x) \varphi(y)  $  yields that  
$f_1 (\Aut^0(A))/  R_u((\Aut^0(A))$   is   
isomorphic to  a subgroup of the group   of     
diagonal matrices  of  $GL_3(k)$  satisfying 
$a_{22} = a_{11} $ and $ a_{3,3}  = a_{11}^ 3 $. 

\

(ii)  Suppose from now on that  $\alpha=\beta =0 $. 
Then, $$ 0 =  \varphi(x)^2 = \lambda_{x,x} (\lambda_{xy,x} +  
\lambda_{yx, x})xyx  +  (2\lambda_{x,x} \lambda_{yxy,x } + \lambda_{xy, x} ^2 +\lambda_{yx, x}^2) xyxy, $$
giving
\begin{equation}\label{3alphax} 
\lambda_{yx, x} = -\lambda_{xy,x},   \  \    \lambda_{yxy,x }= 
-\lambda_{xy, x} ^2\lambda_{x,x} ^{-1}.  \end{equation}
Similarly, 
\begin{equation}\label{3alphay} 
\lambda_{xy, y} = -\lambda_{yx,y},     \  \   \lambda_{xyx,x }= 
-\lambda_{yx, y} ^2\lambda_{y,y} ^{-1}.  \end{equation}

Then, $$ 0 =  \varphi(z) \varphi(x) = 
(\lambda_{z,z} \lambda_{z,x}+\lambda_{x, x} \lambda _{yxy,z}) yxyx, $$
and
$$ 0 =  \varphi(z) \varphi(y) = (\lambda_{z,z} \lambda_{z,y}+ \lambda_{y, y} 
\lambda _{xyx,z}) yxyx, $$
which gives 

\begin{equation} \label{3zcon2}  \lambda_{yxy, z} = -  
\lambda_{z,z}\lambda_{z,x}\lambda_{x,x}^{-1} ,  \  \   \lambda_{xyx, z}= -  
\lambda_{z,z}\lambda_{z,y}\lambda_{y,y}^{-1}
\end{equation}

Conversely, it  is easy to check that  any element of $GL(J)$  satisfying 
the  Equations  \ref{3diag1}, \ref{3diag2}, \ref{3zcon1}, \ref{3zdiag},   
\ref{3alphax}, \ref{3alphay} and  \ref{3zcon2}   and such that the image  
under $\varphi $ of elements   of ${\mathcal B}  \cap J_2$ is    the 
multiplicative  extension of the images of $x, y $ and $z$  of  
$\varphi $ is in $ \Aut^0(A)$.    So,  identifying   as in (i), 
$GL(J_1/J_2) $ with $GL_3(k)$ via the ordered basis 
$\{ x+J_1, y+J_1, z+J_1 \} $  of $ J_1/J_2 $,  
$f_1 (\Aut^0(A)) $ is    the  subgroup  of lower triangular matrices  
$(a_{ij}) $ satisfying $a_{33} = a_{11} a_{22} $, $ a_{2,1}=  0 $,  
$R_u(f_1 (\Aut^0(A)) )$    is the subgroup of  $f_1 (\Aut^0(A)) $ 
consisting of matrices  which satisfy additionally that 
$ a_{1,1}=a_{2,2}= a_{3,3}   =1 $ and   $f_1 (\Aut^0(A))/ R_u(f_1 (\Aut^0(A)) ) 
$ is   isomorphic to the  subgroup of diagonal matrices satisfying 
$a_{3,3}= a_{1,1}a_{2,2}$.  So,  (ii)  follows by Proposition \ref{Outgen}.
\end{proof}

For the next result, note that the centre of an affine   
algebraic group is a closed subgroup of the group, and hence it makes 
sense to speak of the dimension  of the centre as an algebraic variety.

\begin{Proposition} \label{3unipotent}  Suppose that  $A $  satisfies 
Hypothesis  \ref{Ahyp3} and that  the elements $\alpha, \beta $  
Proposition \ref{3nailstrcha} are both equal to $0$.
Then $$\dim(Z(R_u(\Out^0(A))) ) = 3. $$ 
\end{Proposition}

\begin {proof}  Let $\Ub =R_u(\Out^0(A)) $, 
let $\{x+J_2, y+J_2, z+J_2\} $ be  a basis of  
$J_1/J_2 $  as in  Proposition \ref{3nailstrcha}   and    
let ${\mathcal B}$ be the ordered  basis $\{ x, y, z, xy+yx, 
xy-yx,  xyx, yxy, xyx, xyxy \} $  of $J$  (note the  difference  with the 
basis used in  Proposition  \ref{out3}).  Identify $\Aut^0(A)$ with a  
subgroup of $GL_8(k)$   through the basis ${\mathcal B}$.     Let 
$B_8(k)$ be the subgroup of $GL_8(k)$ consisting of lower triangular  
matrices and let $U_8(k) $ be the subgroup of $B_8(k)$  
of strictly unitriangular matrices.  Let $\varphi \in GL_8(k)$  be such that 
that the image  under $\varphi $ of elements   of 
${\mathcal B}  \cap J_2$ is    the 
multiplicative and linear   extension of the images of $x, y $ and $z$  under  
$\varphi $.   By   the proof of Proposition  \ref{out3},  
$\varphi \in\Aut^0(A)$ if and only if 
\begin{equation*} \lambda_{x,y} =\lambda_{y,x}=\lambda_{xy,z}=  
\lambda_{yx, z}= 0;  \end{equation*}
\begin{equation*}\lambda_{xy+yx,x}=\lambda_{xy+yx,y}=\lambda_{xy+yx,z}=
\lambda_{xy-yx,z} =0; \end{equation*}
\begin{equation*}\lambda_{yxy,x}=-\lambda_{xy-yx,x}^2 \lambda_{x,x}^{-1},  
\lambda_{xyx,y}=-\lambda_{xy-yx,y}^2 \lambda_{y,y}^{-1}; \end{equation*} 
\begin{equation*}  \lambda_{yxy, z} = -  
\lambda_{z,z}\lambda_{z,x}\lambda_{x,x}^{-1} ,    \lambda_{xyx, z}= -  
\lambda_{z,z}\lambda_{z,y}\lambda_{y,y}^{-1}. \end{equation*}

In particular,  $\Aut^0(A)$    is   contained 
in  $B_8(k)$ and   $R_u(\Aut^0(A))  = \Aut^0(A) \cap U_8(k)$.   
Thus  $R_u(\Aut^0(A))$   is  the following closed subgroup of $U_8(k)$: 
\[   \left\{ \left( \begin{array}{cccccccc}
1 & 0 & 0 &0 &0 &0 &0 &0 \\
0& 1 &  0 &0 &0 &0 &0 &0 \\
a& e &  1  &0 &0 &0 &0 &0 \\
0& 0 &  0  & 1 &0 &0 &0 &0 \\
b & f & 0 &0 & 1 &0 &0 &0\\
c& -f^2  & -e  & 0 &-2f  & 1 & 0 & 0\\
-b^2 & g &-a  & 0 &-2b &0 & 1 & 0 \\
d&h&i& 2(g+ae+2bf +c)&0 &0&0 &1 
\end{array} \right)   \ : \  a,b,c,d,ef,g, h, i \in k \right\}  \]

We now calculate $\Inn(A)$. Let $ u \in J_1$. So by Lemma \ref{geninn}, 
$\varphi_u (v) =  v + [v, u] + [vu, u] $ for all $v \in V$.
By  Lemma \ref{elemfact2}, $Z\cap J_1 $ is an ideal of $A$, 
hence we may assume that  $ u = \epsilon_1 x + \epsilon_2 y + 
\epsilon_3 (xy -yx) $,  $\epsilon_i \in k$, $ 1\leq i \leq 3 $.
By Lemma \ref{elemfact2} one easily calculates
 $$  \varphi_u (x) = x+ \epsilon_2(xy-yx) +(\epsilon_1\epsilon_2 -2\epsilon_3) 
xyx -\epsilon_2^2yxy; $$
 $$  \varphi_u (y) = y -\epsilon_1(xy-yx) -\epsilon_1^2xyx+
(\epsilon_1\epsilon_2 + 2\epsilon_3) yxy ;  $$
$$  \varphi_u (z) = 1 . $$
Thus $\Inn (A)$ consist of matrices above such that    $ a=e =d=h=i=0 $ and 
$ c+g =-2bf $.

For each $\varphi \in \Aut^0(A)$, let $\varphi^0 :=  \varphi -Id $. Then from 
the above description of $\Aut^0(A)$,  we have that for all 
$\varphi \in \Aut^0(A)$,   $\{xyx, yxy, xyxy \} \subset \Ker(\varphi^0)$, 
$\varphi^0(xy-yx) $,  $\varphi^0(xy+yx) $ and $\varphi^0(z) $ 
are  contained in  the $k$-span of  $\{xyx, yxy, xyxy \}$ and 
$\varphi^0(x) $ and  $\varphi^0(y) $ are contained  in the $k$-span of 
 $\{z, xy-yx, xy+yx, xyx, yxy, xyxy \}$. Hence, 
$\varphi^{\circ} \tau^{\circ }\psi^{\circ} =0 $ for all  
$ \varphi, \tau, \psi  \in \Aut^0(A)$. It follows that  for all 
$\varphi \in \Aut^0(A)$, $\varphi^{-1} = Id+\varphi^{\circ} 
+ {\varphi^{\circ} }^2 $ and for all 
$ \varphi, \tau \in \Aut^0(A)$  
$$\varphi \tau \varphi^{-1}\tau^{-1}  = 
Id + \varphi^{\circ}\tau^{\circ}  - \tau^{\circ}\varphi^{\circ}. $$
Further,  since $\{z, xy-yx, xy+yx, xyx, yxy, xyxy \}$ is contained in the 
kernel of   $\varphi^{\circ}\tau^{\circ} $, 
$$ \varphi^{\circ} \tau^{\circ}(u)  -\tau^{\circ}\varphi^{\circ}(u) =0 $$
for all $ u \in {\mathbb B} $ different from $x, y$.  
Let $\varphi $ be the matrix displayed above and let $\tau $ be   the 
matrix with  $a $ replaced by $a' $ etc. Then,
$$ \varphi^{\circ}  \tau^{\circ}(x) - \tau^{\circ} \varphi^{\circ} (x) = 
(ae'-a'e +2bf'-2b'f)xyx  + (a'i -ai')xyxy, $$
$$ \varphi^{\circ}  \tau^{\circ}(y) - \tau^{\circ} \varphi^{\circ} (y) = 
(ea'-e'a +2fb'-2f'b)yxy  + (e'i -ei')xyxy. $$
Thus,   for any $\varphi \in  \Aut^0(A)$,  
$\varphi \tau \varphi^{-1}\tau^{-1}  \in \Inn (A)$ for all $\tau \in \Aut^0(A)$ 
if and only if $$ a=e =i =0 . $$  

Let $\hat \Zb $ be the   inverse image in $R_{u}(\Aut^0(A)) $ of 
$Z(\Ub)$. There is  is an obvious  injective morphism of varieties   
$ k^6 \to   GL_8(k)$ whose image is  $\hat \Zb$ and   $\hat \Zb$  is a 
closed subgroup  of $\Aut^0(A)$ and hence of $GL_8(k)$.  Thus, 
$\hat \Zb$ has dimension $6$ (see \cite[Corollary 4.3]{humph:book}).
Similarly  $\Inn(A)$ has dimension $3$. So,   
it follows that   $Z(\Ub)$   has dimension $3$ (see \cite[Proposition 7.4B]
{humph:book}).

\end{proof}

Note that if $A$ is  as in Proposition  \ref{2nicequnail} with 
$\alpha=\beta=\gamma = \delta =0 $, then $ A \cong B $.

\begin{Proposition}\label{out2nice}     Suppose that  $A$  satisfies 
Hypothesis \ref{Ahyp2nice} and that $\alpha, \beta, \gamma, \delta $ are   
as in  Proposition \ref{2nicequnail}.  

(i)  If   any one of $\alpha $, $\beta  $, $\gamma $ or $\delta $ 
is non-zero, then $\Out^0(A)/ R_u(\Out^0(A))   $ 
is   contained in a $1$-dimensional torus.

(ii)  If $\alpha= \beta =\gamma =\delta =0 $,  then 
$  \Out^0(A)/ R_u(\Out^0(A))  $  is a two-dimensional torus.

\end{Proposition}

\begin{proof} Let $\{x+J_2, y+J_2\} $ be  a basis of  
$J_1/J_2 $  as in  Proposition \ref{2nicequnail}.  Let  ${\mathcal B}$ 
be the ordered basis $\{x, y,   x^2,  y^2, xy, xyx, yxy, xyxy \}$ of $J_1$.
Let $\varphi \in \Aut(A) $ and  for $u\in {\mathcal B}$ write
$\varphi (u) =\sum_{v \in {\mathcal B} } \lambda_{v, u} v $.    
Then, $$ 0 \equiv \varphi(x)\varphi(y) +\varphi(y)\varphi(x)   
\equiv \lambda_{x,x}  \lambda_{x, y} x^2    +    \lambda_{y,x}  \lambda_{y, y}y^2
\  \mod  \  J_3, $$
So, $$ \lambda_{x,x}  \lambda_{y, x} = \lambda_{x,y}  \lambda_{y, y} =0, $$
that is  either $\lambda_{x,y} =\lambda_{y,x}=0 $ or 
$\lambda_{x,x} =\lambda_{y,y} =0 $.
Identifying $GL ( J_1/J_2)  $ with     $GL_2(k)$    through the  ordered basis 
$\{x  + J_2,  y + J_2 \} $    it follows  that 
$f_1 (\Aut(A) )^0 $  is   contained    
in the subgroup of diagonal matrices. Hence,  if $ \varphi    
\in \Aut^0(A)$ then  
\begin{equation}\label{2nicediag1}   
\lambda_{x,y} =\lambda_{y,x}=0. \end{equation} 
We will assume  from now on  that  this is the case.

(i)  Now suppose that  $\alpha \ne 0 $.  Comparing the coefficient of  $xy^2$
in the relation $\varphi(x)^3 = \alpha \varphi(x)\varphi(y)^2 +\beta 
\varphi(x)^2\varphi(y)^2 $  gives
$$  \lambda_{x,x} ^3 \alpha  = \lambda_{x,x} \lambda_{y,y}^2 \alpha , $$ 
and hence   $ \lambda_{x,x} ^2 =\lambda_{y,y}^2 $.  Thus, identifying  
$f_1 (\Aut(A) )^0 = f_1 (\Aut^0(A) ) $ with   its image in $GL_2(k)$ as above,
$f_1 (\Aut(A) )^0 $  is contained in the  subgroup of diagonal matrices 
$(a_{ij})$   satisfying $ a_{1,1} =a_{2,2} $. This proves (i)  if $\alpha \ne 0 $
and similarly if $\gamma \ne 0 $.   Suppose now  that 
$ \alpha=\gamma =0 $  and $\beta \ne 0 $. 
Comparing  the coefficient of $x^2y^2 $ in the relation 
$\varphi(x)^3  =\beta \varphi(x)^2 \varphi(y)^2 $  gives
$$ \lambda_{x,x}^3 \beta  + 3 \lambda_{x,x}^2\lambda_{y^2, x} =  
\beta \lambda_{x,x}^2 \lambda_{y,y}^2, $$
and hence (since $k$ is of characteristic $3$),
$$\lambda_{x,x}=  \lambda_{y,y}^2 . $$
Thus $f_1 (\Aut(A) )^0 $  is contained in the  subgroup of diagonal matrices 
$(a_{ij})$   satisfying $ a_{1,1} =a_{22}^2 $. This proves (i) if 
$ \alpha=\gamma =0 $, $\beta \ne 0 $.  The proof for the case that 
 $ \alpha=\gamma =0 $, $\delta \ne 0 $ is  similar.

(ii)  Suppose that $\alpha=\beta =\gamma=\delta =0 $.  
The  coefficient of $ x^2y $   in $\varphi(x) \varphi(y) + \varphi(y) \varphi(x)$   is 
$ 2\lambda_{x,x} \lambda_{y^2, y}  =0   $  and the coefficient of $xy^2 $ in    $\varphi(x) \varphi(y) + \varphi(y) \varphi(x)$  is $ 2\lambda_{y,y} \lambda_{x^2, x} $,
hence 
\begin{equation}\label{2nicediag2}    \lambda_{x,x^2}=\lambda_{y,y^2}=0 .
\end{equation}

Now   the coefficient  of  $x^2y^2 $ in  $\varphi(x) \varphi(y) + \varphi(y) \varphi(x)$    is  
$$2\lambda_{x,x}\lambda_{xy^2, y} - 2\lambda_{xy, x}\lambda_{xy, y} + 2 \lambda_{y,y}\lambda_{xy^2, y}  + 2\lambda_{y^2, x}\lambda_{x^2, y},$$
 hence 
\begin{equation}\label{2nicefunny}  \lambda_{xy^2, y}  =( \lambda_{xy,x}\lambda_{xy, y} -\lambda_{x,x}\lambda_{xy^2, y}  -\lambda_{x^2,y}\lambda_{y^2,x} )\lambda_{y,y}^{-1} \end{equation}

Conversely,   any element of $GL(J)$  satisfying 
the  Equations  \ref{2nicediag1}, \ref{2nicediag2}  and \ref{2nicefunny}, 
and such that the image  
under $\varphi $ of elements   of ${\mathcal B}  \cap J_2$ is    the 
obvious multiplicative  extension of the images of $x, y $ and $z$  of  
$\varphi $ is in $ \Aut^0(A)$.    So,  identifying   as in (i), 
$GL(J_1/J_2) $ with $GL_3(k)$ via the ordered basis 
$\{ x+J_1, y+J_1 \} $  of $ J_1/J_2 $,  
$f_1 (\Aut^0(A)) $ is    the  subgroup of diagonal matrices.  
\end{proof}

The next result gives the  dimension  of $Z(R_u(\Out^0(B)))$.

\begin{Proposition} \label{2unipotent}   Suppose that  $A$  satisfies 
Hypothesis \ref{Ahyp2nice} and that  the elements 
$\alpha, \beta, \gamma, \delta $    
of   Proposition \ref{2nicequnail} are all equal to $0$.
Then 
$\dim (Z(R_u(\Out^0(A)))) = 2$. 
\end{Proposition}

\begin {proof}  Let $\{x+J_2, y+J_2\} $ be  a basis of  
$J_1/J_2 $  as in  Proposition \ref{2nicequnail}.  
Let $\Ub= R_u(\Out^0(A))$  and let ${\mathcal B}$ 
be the ordered  basis $\{ x, y,  x^2, y^2, 
xy, x^2y, xy^2, x^2y^2  \} $  of $J$ and identify $\Aut^0(A)$ with a  
subgroup of $GL_8(k)$   through the basis ${\mathcal B}$.    As in the 
Proposition \ref{2unipotent},  let  
$B_8(k)$ be the subgroup of $GL_8(k)$ consisting of lower triangular  
matrices and
let $U_8(k) $ be the subgroup of $B_8(k)$  of strictly unitriangular matrices.
By  Equations, \ref{2nicediag1}, \ref{2nicediag2}  and \ref{2nicefunny},  
$\Aut^0(A)$    is   contained 
in  $B_8(k)$ and hence by Lemma \ref{alguse}(i),   
$R_u(\Aut^0(A))  = \Aut^0(A) \cap U_8(k)$. Further,    
 $R_u(\Aut^0(A))$   is  the following subgroup of $U_8(k)$: 
\[   \left\{ \left( \begin{array}{cccccccc}
1 & 0 & 0 &0 &0 &0 &0 &0 \\
0& 1 &  0 &0 &0 &0 &0 &0 \\
0& e &  1  &0 &0 &0 &0 &0 \\
a & 0 &  0  & 1 &0 &0 &0 &0 \\
b&f&0&0&1&0&0&0\\
i & g & 0 &-e & f &1 &0 &0\\
c& -bf-ae-i & -a  & 0 & b  & 0 & 1 & 0\\
d&h&     -c-b^2& -g-f^2&0 &0&0 &1 
\end{array} \right)   \ : \  a,b,c,d,ef,g, h, i \in k \right\}  \]

We now calculate $\Inn(A)$. Let $ u \in J_1$.   By Lemma \ref{geninn}, 
$\varphi_u (v) =  v + [v, u] + [vu, u] $ for all $v \in V$.
By  Lemma \ref{elemfact2}, $Z\cap J_1 $ is an ideal of $A$, 
hence we may assume that  $ u = \epsilon_1 x + \epsilon_2 y + 
\epsilon_3 xy $,  $\epsilon_i \in k$, $ 1\leq i \leq 3 $.
One easily calculates
 $$  \varphi_u (x) = x  +2 \epsilon_2xy+ (2\epsilon_3 -  2\epsilon_1\epsilon_2) 
x^2y   + 2\epsilon_2^2xy^2; $$
 $$  \varphi_u (y) = y -2\epsilon_1xy   +2\epsilon_1^2x^2y -
2(\epsilon_1\epsilon_2 + \epsilon_3) xy^2. $$

Thus $\Inn (A)$ consist of matrices above such that    $ a=e =d=h=0 $ and 
$ c=-b^2$, $ g =-f^2 $.

For each $\varphi \in \Aut^0(A)$, let $\varphi^0 :=  \varphi -Id $. Then from 
the above description of $\Aut^0(A)$,  we have that for all 
$\varphi \in \Aut^0(A)$,   $\{x^2y, xy^2,   x^2y^2 \} \subset \Ker(\varphi^0)$, 
$\varphi^0(xy) $,  $\varphi^0(x^2) $ and $\varphi^0(y^2) $ 
are  contained in  the $k$-span of  $\{x^y, xy^2, x^2y^2 \}$ and 
$\varphi^0(x) $ and  $\varphi^0(y) $ are contained  in the $k$-span of 
 $\{ x^2, y^2, xy,   x^2y, xy^2, x^2y^2
 \}$. Hence, 
$\varphi^{\circ} \tau^{\circ }\psi^{\circ} =0 $ for all  
$ \varphi, \tau, \psi  \in \Aut^0(A)$. It follows that  for all 
$\varphi \in \Aut^0(A)$, $\varphi^{-1} = Id+\varphi^{\circ} 
+ {\varphi^{\circ} }^2 $ and for all 
$ \varphi, \tau \in \Aut^0(A)$  
$$\varphi \tau \varphi^{-1}\tau^{-1}  = 
Id + \varphi^{\circ}\tau^{\circ}  - \tau^{\circ}\varphi^{\circ}. $$
Further,  since $\{z, xy-yx, xy+yx, xyx, yxy, xyxy \}$ is contained in the 
kernel of   $\varphi^{\circ}\tau^{\circ} $, 
$$ \varphi^{\circ} \tau^{\circ}(u)  -\tau^{\circ}\varphi^{\circ}(u) =0 $$
for all $ u \in {\mathbb B} $ different from $x, y$.  
Let $\varphi $ be the matrix displayed above and let $\tau $ be   the 
matrix with  $a $ replaced by $a' $ etc. Then,
$$ \varphi^{\circ}  \tau^{\circ}(x) - \tau^{\circ} \varphi^{\circ} (x) = 
(ae'-a'e+b'f-bf')x^2y + (a(g'+f'^2)-a'(g+f^2))x^2y^2, $$
$$ \varphi^{\circ}  \tau^{\circ}(y) - \tau^{\circ} \varphi^{\circ} (y) = 
(ea'-e'a +fb'-f'b)xy^2 + (e(c'+b'^2)     -e'(c+b^2))x^2y^2. $$
Thus,   for any $\varphi \in  \Aut^0(A)$,  
$\varphi \tau \varphi^{-1}\tau^{-1}  \in \Inn (A)$ for all $\tau \in \Aut^0(A)$ 
if and only if
$$ a=e =0 \text{\  and \  }  c=-b^2,   g=-f^2 . $$  
Arguing as for Proposition \ref{3unipotent}, 
the inverse image in $R_{u}(\Aut^0(A)) $ of 
$Z(\Ub)$    has dimension $5$,  $\Inn(A)$ has dimension $3$,  
and hence    $Z(\Ub)$   has dimension $2$, as claimed.\end{proof}

\begin{Proposition} \label{Out2bad} Suppose that  $A$  satisfies 
Hypothesis \ref{Ahyp2bad}.  
Then $ \Out^0(A)/ R_u(\Out^0(A))  $ is   contained in a  one-dimensional torus.
\end{Proposition}

\begin{proof}   Let $\{x+J_2, y+J_2\}$ be  a basis of  $J_1/J_2$ satisfying 
the  properties of   Lemmas \ref{2bad} and \ref{2getboring}. Let 
${\mathcal B}$ be the ordered basis 
$\{x,y,  x^2, xy, yx, xyx, yxy, xyxy \} $ of $J_1$  and let  
$\varphi \in \Aut(A)$ be as  before.
 The fact that $\varphi(y)^2 \equiv 0 \mod  \   J_3   $   implies that  
$\lambda_{x,y}= 0 $ and we will assume that this is the case.
Now consider the equation
$$\varphi(x) ^3  \equiv \varphi(y)\varphi(x)\varphi(y) \  \mod \   J_4  $$
$$ \lambda_{x,x}^3 x^3 + \lambda_{x,x}^2\lambda_{y,x} xyx + 2\lambda_{x,x}^2
\lambda_{y,x} x^2y + \lambda_{x,x}\lambda_{y,x}^2 yxy \equiv  \lambda_{y,y}^2 y
(\lambda_{x,x} x + \lambda_{y,x} y ) y  \  \mod  \   J_4, $$
$$ \lambda_{x,x}^2\lambda_{y,x} xyx  +    (\lambda_{x,x}^3 +\lambda_{x,x}
\lambda_{y,x}^2)yxy \equiv \lambda_{y,y}^2\lambda_{x,x}yxy \   \mod  \   J_4,$$
$$\lambda_{x,x}^2\lambda_{y,x}= 0, \  \lambda_{x,x}^3 +\lambda_{x,x}
\lambda_{y,x}^2
=\lambda_{y,y}^2\lambda_{x,x}, $$
$$\lambda_{y,x}= 0, \   \lambda_{x,x}^3=\lambda_{y,y}^2\lambda_{x,x}, $$
$$ \lambda_{y,x}= 0, \lambda_{y,y}= \pm \lambda_{x,x}. $$  The result is 
follows from these calculations and Proposition \ref{Outgen}. 
\end{proof}

\section{Proof of Theorem \ref{stabletomor} and Theorem  
\ref{abeliandefectcon}}
 
\begin{prfstab} There is  a stable equivalence of Morita type between  
$A$ and $B$, by \cite[Th\'eor\`eme  4.15]{rouq:stab}, 
$ \Out^0(A)$ and $\Out^0(B)$ are isomorphic as algebraic groups.  In particular,
$  \Out^0(A)/R_u(\Out^0(A))  \cong \Out^0(B)/ R_u(\Out^0(B))  $    and 
$R_u (\Out^0(A))  \cong  R_u(\Out^0(B)$.   
By  Proposition \ref{out2nice}, and the remark preceding it, 
$\Out^0(B)/R_u(\Out^0(B))$ is a 
$2$-dimensional torus, and by  Proposition \ref{2unipotent}, 
$\dim (Z(R_u(\Out^0(A)))) = 2 $. Hence, it follows  from 
Propositions \ref{out3}, \ref{3unipotent}, \ref{out2nice}, 
and  \ref{Out2bad} that if  $A$ satisfies one of the hypotheses \ref{Ahyp3}, 
\ref{Ahyp2nice} or \ref{Ahyp2bad}, then $A \cong B$. The result follows as 
by Proposition \ref{firstreduction} and Lemma \ref{2start}, 
$A$   does satisfy  one of the hypothesis \ref{Ahyp3}, 
\ref{Ahyp2nice} or \ref{Ahyp2bad}. 
\end{prfstab}

\begin{prfabe} By the structure theory of blocks with 
normal defect groups \cite{Ku},  $C$ is Morita equivalent  to   a
twisted group algebra    $k_{\alpha} P \rtimes E $, where 
$E$ is a $p'$-subgroup of  the automorphism group of $P$ and    
$\alpha $ is an element of   $H^2(E, k^{\times}) $.  Since $A$ is not 
nilpotent and  $C$ has up to isomorphism   only one simple module,   it
is well known (see, for instance,  \cite[Theorem 1.1]{HoKe}), 
that $ E  $ is a Klein-$4$-group   and that 
$C$ is Morita equivalent  
to the algebra $B =k\langle X, Y \rangle/\langle X^3, Y^3, XY +YX\rangle$.

By Rouquier's work \cite[6.3]{rouq:stab}  
(see also \cite[Theorem A.2]{Li}) there is a stable equivalence of Morita 
type between $A$ and $C$.
By \cite{PuUs1} (see also \cite{Keslin09}), 
the blocks  $A$ and $C$ are    perfectly isometric. Hence by \cite{Brou2}, 
the centers of $A$ and $ C$ are  isomorphic  as $k$-algebras and    up to 
isomorphism $A$    has one simple module. Consequently,   the dimension of a 
basic algebra of $A$  is equal to  $9$.
The result follows from Theorem \ref{stabletomor} applied to $B$  
and a basic algebra of  $A$.
\end{prfabe}

{\bf Acknowledgments.}   I would like to thank Meinolf Geck, 
Markus Linckelmann  and Nicole Snashall  for    helpful discussions 
on  various aspects  of this paper.


\bigskip\bigskip

\begin{thebibliography}{WWW}

\bibitem{Brou2}{\sc M.~Brou\'e} 
Isom\'etries parfaites, types de
blocs, cat\'egories d\'eriv\'ees, 
{\em Ast\'erisque},{\bf 181-182} (1990), 61--92.


\bibitem{HoKe} M. Holloway, R. Kessar, {\em Quantum complete rings and blocks
with one simple module}, Q. J. Math. {\bf 56} (2005), 209--221.

\bibitem{humph:book} J.E. Humphreys, Linear Algebraic Groups,  Graduate Texts 
in Mathematics, Springer, (1981).


\bibitem{Keslin09}  R. Kessar, M. Linckelmann, 
On stable equivalences and blocks with one simple module. {\it J. Algebra}  
{\bf 323}  (2010),  no. 6, 1607–1621.

\bibitem{Kiy} M. Kiyota, {\em On $3$-blocks with an elementary abelian
defect group of order $9$}, J. Fac. Sci. Univ. Tokyo  Sect. IA Math.
{\bf 31} (1984), 33--58.



\bibitem{kulshammer:habil} B. K\"ulshammer, Modulare Gruppenalgebren 
und ihre Blo\"cke, Habilitationsschrift, University
of Dortmund, 1984.

\bibitem{Ku} B. K\"ulshammer, {\em Crossed products and blocks with  
normal defect groups} Comm. Algebra {\bf 13} (1985), 147--168.



\bibitem{kulshammer:locsym} B. K\"ulshammer, Symmetric local algebras and 
small blocks of finite groups {\it J. Algebra} {\bf 88} (1984), 190-195. 


\bibitem{konzim} S. K\"onig and A. Zimmermann, Derived Equivalences for 
Group Rings, Lecture Notes in Math. vol. 1685, Springer, Berlin (1998).


\bibitem{Li}  M.Linckelmann, {\em  Trivial source bimodules rings for blocks 
and $p$-permutation equivalences}, {\it Trans. Amer. Math. Soc.}, 
{\bf 361} (2009), 1279-1316.

\bibitem {puignilp} L. Puig, Nilpotent blocks and their source algebras. 
{\it Invent. Math. } {\bf 93} (1988),  no. 1, 77--116. 


\bibitem{PuUs1} L. Puig, Y. Usami, {\em Perfect isometries for blocks
with abelian defect groups and Klein four inertial quotients},
J. Algebra {\bf 160} (1993), 192--225.


\bibitem {rickstab} J.Rickard, Derived categories and stable equivalence,
{\it J. Pure Appl. Algebra} {\bf 61} (1989),  no. 3, 303--317.

\bibitem {ricksplen} J.Rickard, Splendid equivalences: derived categories 
and permutation modules. {\it Proc. London Math. Soc. (3)} {\bf 72}  (1996),  
no. 2, 331--358. 


\bibitem {rouquier-glue}  R. Rouquier, Block theory via stable and derived 
equivalences, Modular representation theory of finite groups
(Charlottesville, {VA}, 1998),101--146, de Gruyter, Berlin (2001).

\bibitem {rouq:stab} R. Rouquier, Automorphismes, graduations et 
cat\'egories triangul\'ees, arXiv: 1008.1971. 



\end{thebibliography}
\end{document}